\documentclass[12pt]{amsart}

\usepackage{latexsym, amssymb, amsmath,ulem,soul,esint}

\usepackage{amsfonts, graphicx}
\usepackage{graphicx,color}
\usepackage{url}
\newcommand{\be}{\begin{equation}}
\newcommand{\ee}{\end{equation}}
\newcommand{\beq}{\begin{eqnarray}}
\newcommand{\eeq}{\end{eqnarray}}

\newtheorem{thm}{Theorem}[section]
\newtheorem{conj}{Conjecture}[section]

\newtheorem{lma}{Lemma}[section]
\newtheorem{prop}{Proposition}[section]

\theoremstyle{remark}
\newtheorem{rem}{Remark}[section]
\numberwithin{equation}{section}
\def\tr{\operatorname{tr}}

\def\be{\begin{equation}}
\def\ee{\end{equation}}
\def\bee{\begin{equation*}}
\def\eee{\end{equation*}}

\newcommand{\de}{\partial}

\newcommand{\Ric}{\mathrm{Ric}}
\newcommand{\ov}[1]{\overline{#1}}

\newcommand{\vol}{\mathrm{vol}}

\newcommand{\ve}{\varepsilon}

\newcommand{\loc}{\mathrm{loc}}

\newcommand{\tinabla}{\tilde{\nabla}}

\def\K{K\"ahler }

\def\Ric{\text{\rm Ric}}
\def\Rm{\text{\rm Rm}}

\def\tr{\operatorname{tr}}

\def\e{\varepsilon}
\def\a{{\alpha}}
\def\b{{\beta}}

\begin{document}

\title{Ricci-Deturck flow from rough metrics and applications}

\author[J. Chu]{Jianchun Chu}
\address[Jianchun Chu]{School of Mathematical Sciences, Peking University, Yiheyuan Road 5, Beijing, P.R.China, 100871}
\email{jianchunchu@math.pku.edu.cn}

\author[M.-C. Lee]{Man-Chun Lee}
\address[Man-Chun Lee]{Department of Mathematics, The Chinese University of Hong Kong, Shatin, N.T., Hong Kong}
\email{mclee@math.cuhk.edu.hk}


\date{\today}

\begin{abstract}
Motivated by the recent work of Lamm and Simon, in this work we study the short-time existence theory of Ricci-Deturck flow starting from  rough metrics which are bi-Lipschitz and have small local scaling invariant gradient concentration.  As applications, we use it to study several stability problems related to the scalar curvature.
\end{abstract}


\maketitle

\section{Introduction}

On a Riemannian manifold, the Ricci flow $g(t)$ is a family of metrics evolving along its Ricci direction:
\begin{equation}
\partial_t g(t)=-2\Ric(g(t)).
\end{equation}
Since its introduction by Hamilton \cite{Hamilton1982}, the Ricci flow has been used in a wide variety of settings to regularize metrics.  From analytic point of view, the Ricci flow is only weakly parabolic unless special gauge is used. It was later discovered by Deturck \cite{Deturck1983} that the Ricci flow is diffeomorphic to a strictly parabolic system which by now called the Ricci-Deturck flow.

For a given smooth Riemannian metric $h$ on $M$, a smooth family $g(t),t\in [a,b]$ is said to be a solution to the Ricci-Deturck $h$-flow if it satisfies
\begin{equation}\label{eqn: h-flow-Ricciform}
\left\{\begin{array}{ll}
\partial_t g_{ij}=-2\Ric_{ij}+\nabla_i W_j +\nabla_jW_i;\\[1mm]
W^k=g^{pq}\left(  \Gamma_{pq}^k-\tilde\Gamma_{pq}^k\right)
\end{array}
\right.
\end{equation}
where $\Gamma$ and $\tilde \Gamma$ denote the Christoffel symbols of Levi-Civita connections of $g(t)$ and $h$ respectively.  In local coordinate, \eqref{eqn: h-flow-Ricciform} is equivalent to the following strictly parabolic system:
\begin{equation}\label{eqn: h-flow}
\begin{split}
\partial_t g_{ij}&=g^{pq} \tilde \nabla_p\tilde\nabla_q g_{ij}-g^{kl}g_{ip}h^{pq} \tilde R_{jkql}-g^{kl}g_{jp} h^{pq}\tilde R_{ikql}\\
&\quad +\frac12 g^{kl}g^{pq}\big(\tilde \nabla_i g_{pk}\tilde \nabla_j g_{ql}+2\tilde \nabla_k g_{jp}\tilde \nabla_q g_{il}-2\tilde \nabla_k g_{jp}\tilde \nabla_l g_{iq}\\[1mm]
&\quad -2\tilde \nabla_j g_{pk}\tilde \nabla_l g_{iq}-2\tilde \nabla_i g_{pk}\tilde \nabla_l g_{qj}\big)
\end{split}
\end{equation}
where $\tilde{R}$ and $\tilde \nabla$ denote the curvature and the Levi-Civita connection of $h$.  The flow is strictly parabolic and is known to be diffeomorphic to the Ricci flow as seen from \eqref{eqn: h-flow-Ricciform}. We follow the terminology in \cite{LammSimon2021} to call it the Ricci-Deturck $h$-flow in order to emphasis its dependence on the chosen background metric $h$. Since the flow is strictly parabolic, it is expected that the flow will improve the regularity of the initial data.  When the initial data is sufficiently regular at infinity, it is not difficult to construct a regular solution for a short time using implicit function theorem with some suitable Banach space. However without regularity assumption at infinity, the existence problem becomes difficult. The first important result is by Shi \cite{Shi1989} where he established a short-time solution to the Ricci-Deturck flow with $h=g_0$ where $g_0$ is complete with bounded curvature. Using this, Shi was able to construct complete bounded curvature Ricci flow from bounded curvature initial metric. Moreover, the derivative of the curvature become bounded for positive time. Since then, there are a large body of works improving Shi's result, for example see \cite{KochLamm2012,Hochard2016,BCRW2019,Lai2019,ChauLee2020,LeeTam2020,Lai2021,LeeTam2021-GT,SimonTopping2017,ToppingHao2021} and the references therein.

Motivated by the question concerning the singular metrics in scalar curvature problem (for instances see Conjecture~\ref{conj-1}),  in this work we are primarily interested in the case when the initial metric is in $L^\infty$. We say that $g$ is a $L^\infty$ metric if $g$ is a measurable section of $\mathrm{Sym}_2(T^*M)$ such that $\Lambda^{-1}h\leq g_0\leq \Lambda h$ almost everywhere on $M$ for some $\Lambda>1$ and smooth metric $h$ on $M$.  To the best of authors' knowledge, the first regularization result along this line is by Simon \cite{Simon2002} where he considered the case when $g_0$ is close to a smooth bounded curvature metric $h$ in $L^\infty$ sense. In particular, it was shown that if $(1-\e)h\leq g_0\leq (1+\e)h$  for some sufficiently small dimensional constant $\e$, then the Ricci-Deturck $h$-flow admits a short-time solution on $M\times (0,T]$ which is smooth with uniform higher order regularity for $t>0$. Moreover, if $g_0$ is $C^0_{\loc}$, then the resulting solution attains $g_0$ in $C^0_{\loc}$ as $t\to 0$, see also \cite{KochLamm2012} for a related works on Euclidean space using heat kernel.  The Ricci flow theory from $C^0$ initial metrics has been developed further by Burkhardt-Guim \cite{Burkhardt2019} in order to study the notion of scalar curvature lower bound for $C^0$ metrics on compact manifolds.  The theory of Ricci-Deturck flow from $C^0$ data turns out to be very powerful in many stability problem related to the scalar curvature lower bound, for instances see \cite{McFeronSze2012,Bamler2016,JiangShengZhang2021,LeeTam2021,ShiTam2018}. It is therefore important to seek for extension to more general $L^\infty$ initial data.

It is however hard to obtain estimates under only $L^\infty$ data in general. In a recent work by Lamm and Simon \cite{LammSimon2021},  rough metrics $g_0\in L^\infty\cap W^{2,2}$ on complete four-manifolds $(M^4,h)$ are considered where $h$ is assumed to have bounded geometry of infinite order.   It was shown that the Ricci-Deturck $h$-flow exists for short-time and is uniformly smooth for $t>0$. Moreover, the flow converges back to the initial data in $W^{2,2}_{\loc}$ sense as $t\to 0$. We would like to stress out that a $L^\infty\cap W^{1,n}_{\loc}$ Riemannian metric in general not continuous, see the example discussed in the introduction of \cite{LammSimon2021}. Motivated by their works, we consider the case when $g_0\in L^\infty\cap W^{1,n}$. The following is the quantitative version of the existence theorem, which generalizes the works in \cite{LammSimon2021} to weaker initial regularity assumptions and higher dimensions.  We always assume $n>2$, where $n$ is the dimension of $M$.

\begin{thm}\label{thm-existence}
Let $(M^n,h)$ be a complete Riemannian manifold with $|\Rm(h)|\leq 1$.  For any $\Lambda_0>1$,  there is $\e_0(n,\Lambda_0)>0$ such that the following holds.  Suppose $g_0$  is a $L^\infty\cap W^{1,n}_{\loc}$ Riemannian metric, not necessarily smooth, on $M$ such that
\begin{enumerate}
\item[(i)] $\Lambda_0^{-1}h\leq g_0\leq \Lambda_0 h$ on $M$;
\item[(ii)] for all $x\in M$,
$$\left(\fint_{B_h(x,1)} |\tilde \nabla g_0|^n \; d\mathrm{vol}_h \right)^{1/n}\leq \e$$
for some $\e<\e_0$.
\end{enumerate}
Then there are $T(n,\Lambda_0),C_0(n,\Lambda_0),C_{n}(n),\Upsilon_n>0$ and a smooth solution $g(t)$ to the Ricci-Deturck $h$-flow on $M\times (0,T]$ such that
\begin{enumerate}\setlength{\itemsep}{0.5mm}
\item[(a)] $(\Upsilon_n\Lambda_0)^{-1} h\leq g(t)\leq \Upsilon_n\Lambda_0 h$ on $M\times (0,T]$;
\item[(b)]  For all $x\in M$ and $t\in (0,T]$,
$$\left(\fint_{B_h(x,1)} |\tilde \nabla g(t)|^n \; d\mathrm{vol}_h \right)^{1/n}\leq   C_n\e+C_0 t^{1/n};$$
\item[(c)] For any $k\in \mathbb{N}$, there is $C(k,n,\Lambda_0)>0$ such that for all $t\in (0,T]$,
\begin{equation*}
\sup_M |\tilde\nabla^k g(t)|\leq C(k,n,\Lambda_0) t^{-k/2}.
\end{equation*}
\item[(d)] For all $t\in (0,T]$,  $\sup_M t^{1/2} |\tilde\nabla g(t)|\leq C(n,\Lambda_0)  \sqrt{\e+t^{1/n}}$;
\item[(e)] If $g_0\in C^\infty_{\loc}(\Omega)$ for some $\Omega\Subset M$, then $g(t)\to g_0$ in $C^\infty_{\loc}(\Omega)$ as $t\to 0$;
\item[(f)] $g(t)\to g_0$ in $W^{1,n}_{\loc}$ as $t\to 0$.
\end{enumerate}
\end{thm}

Moreover, the solution is unique within the suitable class, see Theorem~\ref{thm-unique}. Here the background metric $h$ is serving as a reference to rule out the bubbling singularity in view of concentration only. A simple scaling argument will show that one can replace $h$ by any other fixed metric which is uniformly $C^1$ comparable to $h$.

In short, Theorem~\ref{thm-existence} shows that $L^\infty$ metric with small local gradient concentration can be smoothed with quantitative estimates. In particular, this theorem implies that if $g_0\in L^{\infty}\cap W^{1,n}$ and $h$ is uniformly non-collapsed, then one can deform the metric $g_0$ slightly to a smooth metric. We refer readers to Theorem~\ref{shorttime-compact} for the detailed statement. This is in spirit similar to the work in \cite{ChanChenLee2022} where smooth metrics with small $L^{n/2}$ curvature are considered.

\begin{rem}
Since the Ricci-Deturck $h$-flow is diffeomorphic to the Ricci flow, it is nature to expect that we can  produce Ricci flow starting from $g_0$. If $g_0$ is assumed to be smooth, this was carried out in \cite{HuangTam2018}. For rough initial data, we refer readers to \cite{LammSimon2021} for the detailed exposition. We do not pursue here.
\end{rem}

When the background metric $(M,h)$ is standard Euclidean space, we are able to show that the constructed Ricci-Deturck $h$-flow exists for all time. Moreover, it will converge to $h$ if it is initially close to $h$ at spatial infinity, see Section~\ref{Sec:stability}.

The most important feature of the Ricci-Deturck flow (or Ricci flow) is that it tends to preserve non-negativity of curvature. In particular, it is well-known that the Ricci-Deturck flow preserved the scalar curvature lower bound in the smooth case. Using this property and the above new existence theory, we study several problems related to metrics with scalar curvature lower bound.

When $M=\mathbb{T}^n$, the celebrated work of Schoen and Yau \cite{SchoenYau1979,SchoenYau1979-2}, Gromov and Lawson \cite{GromovLawson1980} stated that metrics with $\mathcal{R}\geq 0$ must be flat. In \cite{Gromov2014}, Gromov asked if a sequence of metrics $g_i$ with $\mathcal{R}(g_i)\geq -i^{-1}$ will sub-converge to a flat metric in some appropriate weak sense. In Theorem~\ref{thm:stability-sigma}, we show that if $g_i$ is uniformly $L^\infty$ with uniformly small gradient concentration, then it will sub-converge to a flat metric in $L^p$ for all $p>0$. For related works, we refer interested readers to \cite{AHPPS2019,Allen2020,CAP2020,ChuLee2021,ChuLee2021-2} and the survey paper of Sormani \cite{Sormani2021} for a comprehensive discussion.

The torus stability problem is on the other hand closely related to a rigidity conjecture of Schoen predicting that $L^\infty$ metrics on $\mathbb{T}^n$ with co-dimension three singularity and $\mathcal{R}\geq 0$ on regular part must have removable singularity, see Conjecture~\ref{conj-1} for the full conjecture. There are many important works toward the full conjecture. For instance, it was confirmed by Li and Mantoulidis \cite{LiMantoulidis2019} when $n=3$ using minimal surface method. When the singular metric is continuous, it was confirmed by the second named author and Tam \cite{LeeTam2021}. To this end, in this work we consider the singular metric which is additionally $W^{1,n}$ and show that they must be Ricci flat on the regular part, see Theorem~\ref{rigidity-FollowingLEETAM}. Using the same argument as in \cite{LeeTam2021}, we also prove a positive mass theorem for metrics with analogous singularity, see Theorem~\ref{rigidity-FollowingLEETAM-PMT}.

Another related problem is whether scalar curvature lower bound is preserved under weak convergence if the limit is a-priori smooth. The first result along this line is by Gromov \cite{Gromov2014} saying that the scalar curvature lower bound will be preserved under $C^0$ convergence. In \cite{Bamler2016}, a Ricci flow proof was given by Bamler. This was later generalized by Burkhardt-Guim \cite{Burkhardt2019} to the case when the limit is not smooth. Using the Ricci-Deturck flow and uniqueness, we show that the same conclusion holds for uniformly bi-Lipschitz sequence converging to a smooth metric in $W^{1,n}$, see Theorem~\ref{thm:stability-R-lowerbound}.

The paper is organized as follows. In Section~\ref{Sec:apriori}, we will establish a-priori estimates of smooth Ricci-Deturck flow. In Section~\ref{Sec:Existence}, we will use these a-priori estimates to construct short-time solution from rough initial data. In particular, we prove our main result, Theorem~\ref{thm-existence}. In Section~\ref{Sec:stability}, we will study the stability of rough initial metric on Euclidean along the Ricci-Deturck flow. In Section~\ref{Sec:almostRid}, we will use the Ricci-Deturck flow to study the stability problem in scalar curvature related to the Yamabe type of manifolds. In Section~\ref{Sec:singular}, we use the Ricci-Deturck flow to regularize singular metrics and establish rigidity. In Section~\ref{Sec:scalarPres}, we study the scalar curvature persistence problem under $W^{1,n}$ convergence using the new existence theory.

\medskip

{\it Acknowledgement}:
The authors would like to thank Tobias Lamm for suggesting the problem and useful discussion. J. Chu was partially supported by Fundamental Research Funds for the Central Universities (No. 7100603592 and No. 7100603624).

\section{A-priori estimates}\label{Sec:apriori}
In this section, we will assume the Ricci-Deturck $h$-flow to be smooth up to $t=0$, and derive a-priori estimates under small gradient concentration assumption. To fix our notations, we will use $|\Omega|$ to denote the measure of $\Omega$ with respect to $h$. We will also denote $a\wedge b=\min\{a,b\}$ for $a,b\in \mathbb{R}$.

\medskip

Suppose $(M,h)$ is a complete (not necessarily compact) manifold with bounded curvature. The following proposition shows that we may without loss of generality assume that all derivative of curvature $\Rm(h)$ are bounded.
\begin{prop}\label{prop:RF-smoothing}
Suppose $(M,h_0)$ is a complete manifold with $|\Rm(h_0)|\leq 1$, then for any $\e>0$, there is a complete metric $\tilde h$ on $M$ such that
\begin{enumerate}
\item[(i)] $\|\tilde h-h_0\|_{C^1(M,h_0)}<\e$;
\item[(ii)] for all $k\in \mathbb{N}$, there is $C(n,k,\e)>0$ such that
$$\sup_M|\nabla^k \Rm(\tilde h)|\leq C(n,k,\e).$$
\end{enumerate}
\end{prop}
\begin{proof}
Since $h$ has bounded curvature, the classical existence theory of Ricci flow by Shi \cite{Shi1989} and doubling time estimates infer that there is $T_n>0$ and a solution to the Ricci flow $h(t)$ starting from $h(0)=h_0$ such that $|\Rm(h(t))|\leq 2$ on $[0,T_n]$. Shi's estimate implies that for all $k\in \mathbb{N}$, there is $C(n,k)>0$ such that for all $t\in (0,T_n]$,
\begin{equation}\label{prop:RF-smoothing eqn 1}
|\nabla^k \Rm(h(t))|\leq \frac{C(n,k)}{t^{k/2}}.
\end{equation}
Since $|\Rm(h(t))|\leq 2$ on $[0,T_n]$, the Ricci flow equation implies
\begin{equation}
\left|\frac{\de}{\de t}h(t)\right| \leq 2|\Ric(h(t))| \leq 4
\end{equation}
and so
\begin{equation}\label{prop:RF-smoothing eqn 2}
e^{-4t}h_0\leq h(t)\leq e^{4t}h_0.
\end{equation}
It remains to consider the $C^1$ continuity of $h(t)$. By Shi's estimates,
\begin{equation}
\begin{split}
\frac{\de}{\de t}|\nabla^{h_0}h(t)|_{h_0}^2&=4\langle \nabla^{h_0}\Ric(h(t)), \nabla^{h_0} h(t) \rangle_{h_0}\leq \frac{C_1}{t^{1/2}}|\nabla^{h_0}h(t)|_{h_0}.
\end{split}
\end{equation}
For any $\sigma>0$, we define $Q=\sqrt{|\nabla^{h_0}h(t)|_{h_0}^2+\sigma}$. Then the above shows $\de_{t}(Q^{2})\leq C_{1}Qt^{-1/2}$ and then $\de_{t}Q\leq C_{1}t^{-1/2}$. Integrating on both sides on $[0,t]$ and using $|\nabla^{h_0}h(0)|_{h_0}=0$,
\begin{equation}
\sqrt{|\nabla^{h_0}h(t)|_{h_0}^2+\sigma}-\sigma \leq C_{1}t^{1/2}.
\end{equation}
Letting $\sigma\to0$, we conclude that $|\nabla^{h_0}h(t)|_{h_{0}}\leq C(n,k)t^{1/2}$ on $M\times [0,T_n]$. Combining this with \eqref{prop:RF-smoothing eqn 1} and  \eqref{prop:RF-smoothing eqn 2}, the assertion follows by choosing $\tilde h=h(t_\e)$ for sufficiently small $t_{\e}$.
\end{proof}

By Proposition~\ref{prop:RF-smoothing} and scaling, in what follows we may as well assume $h$ to satisfy the following slightly stronger assumption:
\begin{equation*}
(\star):\left\{
\begin{array}{ll}
\forall k\in \mathbb{N},\; \exists\; C(n,k)>0, \; \sup_M|\nabla^k \Rm(h)|\leq C(n,k);\\[1mm]
\mathrm{diam}(M,h)>4.
\end{array}
\right.
\end{equation*}
If $M$ is non-compact, the diameter is understood to be $+\infty$.  We introduce this so that the Sobolev inequality below is always local.  It is standard (see e.g. \cite[Theorem 14.3]{LiBook}) that a local Sobolev inequality holds with respect to the metric $h$: there is $C_{S}>0$ such that for any $x_0\in M$ and $\phi\in C^\infty_{c}(B_h(x_0,2))$,
\begin{equation}\label{Sobo-ineq}
\left(\fint_{B_h(x_0,2)} \phi^\frac{2n}{n-2} d\mathrm{vol}_h \right)^\frac{n-2}{n}\leq C_S \cdot \left(\fint_{B_h(x_0,2)} |\tilde\nabla \phi|^2 d\mathrm{vol}_h \right).
\end{equation}
Since we have fixed the radius and rescaled the curvature, the Sobolev constant is indeed a dimensional constant. We also remark that the Sobolev inequality is stable with respect to $C^0$ perturbation of metrics.

\medskip

Before we establish the crucial a-priori estimates, we will collect an important smoothing estimates in Ricci-Deturck $h$-flow.

\begin{lma}\label{lma:localEstimate-fromC^1}
Suppose $g(t)$ is a smooth solution of the Ricci-Deturck $h$-flow on $M\times [0,T]$ for some background metric $h$ satisfying $(\star)$ so that for $t\in (0,T]$,
\begin{equation}
\left\{
\begin{array}{ll}
\Lambda^{-1} h\leq g(t)\leq \Lambda h;\\[1mm]
\sup_M |\tilde \nabla g(t)|\leq C_1 t^{-1/2}
\end{array}
\right.
\end{equation}
for some $\Lambda,C_1>0$. Then for all $k\in \mathbb{N}$, there is $C(n,k,C_1,\Lambda)>0$ such that for all $t\in (0,T]$,
\begin{equation}
\sup_M|\tilde \nabla^k g(t)|\leq \frac{C(n,k,C_1,\Lambda)}{t^{k/2}}.
\end{equation}
\end{lma}
\begin{proof}

This follows from the proof of \cite[Lemma 4.2, Theorem 4.3]{Simon2002}. We remark here that since we have assumed a bound on $|\tilde\nabla g(t)|$ and metric equivalence,  the ``$\delta$ fair to $h$" assumption as in \cite[Lemma 4.2, Theorem 4.3]{Simon2002} is not necessary.  The proof can be carried over by the Bernstein-Shi's trick, see also \cite{Shi1989}.
\end{proof}

In particular, Lemma~\ref{lma:localEstimate-fromC^1} reduces the question of obtaining derivatives estimates to obtaining $C^1$ estimate of the Ricci-Deturck $h$-flow. The following lemma shows that as long as a concentration estimates of $\tilde \nabla g(t)$ is preserved along the flow, the $C^{1}$ estimates will follow.

\begin{lma}\label{lma:localEstimate-fromC^1conc}
 For any $\Lambda>1$, there is $c_1(n,\Lambda)>0$ such that if $g(t)$ is a smooth solution of the Ricci-Deturck $h$-flow on $M\times [0,T]$ for some background metric $h$ satisfying the assumption $(\star)$ so that
\begin{enumerate}
\item[(i)] $\Lambda^{-1}h\leq g(t)\leq \Lambda h$ on $M\times [0,T]$;
\item[(ii)] $\forall x\in M$ and $t\in (0,T]$,
$$\left(t^{n/2}\cdot \fint_{B_h(x,\sqrt{t})}|\tilde \nabla g(t)|^n d\mathrm{vol}_h \right)^{1/n}\leq c_1\e^2 \quad\text{for some}\; \e\in (0,1);$$
\item[(iii)] $\sup_{M\times [0,T]} |\tilde \nabla^k g(t)|<+\infty$ for all $k\in \mathbb{N}$.
\end{enumerate}
Then for all $t\in (0,T\wedge 1]$, we have
\begin{equation}
\sup_M |\tilde \nabla g(t)|< \frac{\e}{\sqrt{t}}.
\end{equation}
\end{lma}

\begin{proof}
Since the assumptions are invariant under parabolic scaling: $\hat g(t)=\lambda^2 g(\lambda^{-2}t)$ and $\hat h=\lambda^2 h$ for $\lambda\geq 1$, we may assume $T=1$.

Let $c_1$ be a small constant to be chosen. By assumption (iii), for the constant $c_1$, we may let $T_1\in (0,1]$ be the maximal time such that the conclusion holds on $(0,T_1)$. If $T_1<1$, then the conclusion fails at $t=T_1$. We claim that if $c_1$ is sufficiently small, then this will be impossible for any $\e\in (0,1)$.

Suppose on the contrary that $T_1<1$, then there is $x_0\in M$ such that
\begin{enumerate}
\item $|\tilde\nabla g(x_0,T_1)|>\frac12 \e T_1^{-1/2}$;
\item  $\sup_{M}|\tilde \nabla g(t)|\leq \e t^{-1/2}$ for all $t\in (0,T_1]$.
\end{enumerate}
By considering $T_1^{-1}g(T_1t)$ and $T_1^{-1}h$,  we may assume that
\begin{enumerate}
\item[(a)] $|\tilde\nabla g(x_0,1)|>\frac12 \e $;
\item[(b)] $\sup_{M}|\tilde \nabla g(t)|\leq \e t^{-1/2}$ for all $t\in (0,1]$.
\end{enumerate}
We remark here that since $T_1<1$, the rescaled background metric still satisfies $(\star)$. Applying Lemma~\ref{lma:localEstimate-fromC^1}, we deduce that $|\tilde\nabla^2g(t)|\leq C_2t^{-1}$ for some $C_2(n,\Lambda)>0$. In particular, for all $x\in B_h(x_0,\delta)$ with $\delta<1$,
\begin{equation}
|\tilde \nabla g(x,1)|\geq |\tilde \nabla g(x_0,1)|-C_2 \delta\geq \frac14 \e,
\end{equation}
provided that we choose $\delta=\frac1{4C_2}\e=c_3\e$ for some small $c_3(n,\Lambda)>0$. Hence,
\begin{equation}
\begin{split}
 \fint_{B_h(x_0,1)}|\tilde \nabla g(1)|^n d\mathrm{vol}_h
&\geq \frac{|B_h(x_0,\delta)|}{|B_h(x_0,1)|}\fint_{B_h(x_0,\delta)}|\tilde \nabla g(1)|^n d\mathrm{vol}_h \\
&\geq  \frac{|B_h(x_0,c_3\e)|}{|B_h(x_0,1)|} \cdot \frac{\e^{n}}{4^{n}} \\[1.5mm]
&\geq \left( c_4 \e^2\right)^n
\end{split}
\end{equation}
for some $c_4(n,\Lambda)>0$. Here we have used the volume comparison theorem to control the volume ratio up to a fixed scale. On the other hand, assumption (ii) implies that
\begin{equation}
\fint_{B_h(x_0,1)}|\tilde \nabla g(1)|^n d\mathrm{vol}_h \leq \left( c_1 \e^2\right)^n
\end{equation}
which is impossible if we choose $c_1(n,\Lambda)=\frac12 c_4>0$.
\end{proof}

By Lemma~\ref{lma:localEstimate-fromC^1conc}, the existence time can be characterized by the gradient concentration and the bi-Lipschitz estimates along the flow. The following Lemma shows that if the gradient concentration is small, then the local $W^{1,n}$ norm can be estimated by the initial value quantitatively.

\begin{lma}\label{borderline-C1}
For any $\Lambda>1$, there is $\delta_1(n,\Lambda)>0$ such that the following holds. Suppose $g(t)$ is a smooth solution to the Ricci-Deturck $h$-flow on $M\times [0,T]$ for some background metric $h$ satisfying $(\star)$ so that
\begin{enumerate}
\item[(i)] $\Lambda^{-1} h\leq g(t)\leq \Lambda h$ on $M\times [0,T]$;
\item[(ii)] $\forall x\in M$ and $t\in (0,T]$,
$$\left(\fint_{B_h(x,1)} |\tilde\nabla g(t)|^n d\mathrm{vol}_h\right)^{1/n} \leq \delta_{1};$$
\end{enumerate}
then there is $L_1(n,\Lambda)>0$ such that for all $x\in M$ and $t\in [0,T]$,
\begin{equation}
\int_{B_h(x,1)} |\tilde\nabla g(t)|^n d\mathrm{vol}_h \leq \int_{B_h(x,2)} |\tilde\nabla g_0|^n d\mathrm{vol}_h+L_1|B_h(x,2)| \cdot t.
\end{equation}
Here $|\Omega|$ denotes the measure of $\Omega$ with respect to $h$.
\end{lma}
\begin{proof}
In what follows, we will use $C_i$ to denote constants depending only on $n$ and $\Lambda$. Along the Ricci-Deturck $h$-flow, direct computation and the Cauchy-Schwarz inequality show that
\begin{equation}
\frac{\de}{\de t}|\tinabla g|^{2}+\Lambda^{-1}|\tinabla^{2}g|^{2}
\leq g^{ab}\tinabla_{a}\tinabla_{b}|\tinabla g|^{2}+C_1|\tinabla g|^{4}+C_1.
\end{equation}
Since $|\tinabla g|$ may be not smooth, then for sufficiently small $\sigma\in(0,\delta_{1})$, we define the smooth function $Q$ by
\begin{equation}
Q = \sqrt{|\tinabla g|^{2}+\sigma^{2}}.
\end{equation}
It is clear that $|\tinabla Q|\leq|\tinabla^{2}g|$. Then
\begin{equation}\label{W 1 n estimate eqn 1}
\frac{\de}{\de t}Q^{2}+\Lambda^{-1}|\tinabla Q|^{2} \leq g^{ab}\tinabla_{a}\tinabla_{b}Q^{2}+C_2(Q^{2}+1).
\end{equation}
Let $\eta$ be a cut-off function on $M$ such that $\eta\equiv1$ on $B_{h}(x,1)$, $\eta\equiv0$ on $M\setminus B_{h}(x,2)$ and $|\tilde\nabla \eta|\leq 10^4$. Then \eqref{W 1 n estimate eqn 1} implies that
\begin{equation}\label{W 1 n estimate eqn 2}
\begin{split}
& \int_{M}\eta^{4}Q^{n-2}\cdot \partial_tQ^{2}d\mathrm{vol}_{h}
+\Lambda^{-1}\int_{M}\eta^{4}Q^{n-2}|\tinabla Q|^{2}d\mathrm{vol}_{h} \\
\leq {} & \int_{M}\eta^{4}Q^{n-2}g^{ab}\tinabla_{a}\tinabla_{b}Q^{2}d\mathrm{vol}_{h}+C_3\int_{M}\eta^{4}(Q^{n}+1)d\mathrm{vol}_{h}.
\end{split}
\end{equation}
For the first term on the right hand side, integrating by parts and using the Cauchy-Schwarz inequality,
\begin{equation}
\begin{split}
& \int_{M}\eta^{4}Q^{n-2}g^{ab}\tinabla_{a}\tinabla_{b}Q^{2}d\vol_{h} \\
= {} & -\int_{M}\tinabla_{a}(\eta^{4}Q^{n-2}g^{ab})\tinabla_{b}Q^{2}d\vol_{h} \\
\leq {} & -2(n-2)\Lambda^{-1}\int_{M}\eta^{4}Q^{n-2}|\tinabla Q|^{2}d\vol_{h}
+C_4\int_{M}(\eta^{4}Q^{n}+\eta^{3}Q^{n-1})|\tinabla Q|d\vol_{h} \\
\leq {} & -(n-2)\Lambda^{-1}\int_{M}\eta^{4}Q^{n-2}|\tinabla Q|^{2}d\vol_{h}+C_5\int_{M}(\eta^{4}Q^{n+2}+\eta^{2}Q^{n})d\vol_{h}.
\end{split}
\end{equation}
Substituting this into \eqref{W 1 n estimate eqn 2},
\begin{equation}\label{W 1 n estimate eqn 3}
\begin{split}
& \frac{2}{n}\cdot\frac{\de}{\de t}\left(\int_{M}\eta^{4}Q^{n}d\vol_{h}\right)+\frac1{2\Lambda}\int_{M}\eta^{4}Q^{n-2}|\tinabla Q|^{2}d\vol_{h} \\
\leq {} & C_6\int_{M}(\eta^{4}Q^{n+2}+\eta^{2}Q^{n}+\eta^{4}Q^{n}+\eta^{4})d\vol_{h}.
\end{split}
\end{equation}
Since $\eta$ is compactly supported on $B_h(x,2)$, the Sobolev inequality \eqref{Sobo-ineq} implies
\begin{equation}\label{W 1 n estimate eqn 4}
\begin{split}
& \left(\int_{M}(\eta^{4}Q^{n})^{\frac{n}{n-2}}d\vol_{h}\right)^{\frac{n-2}{n}}\\[1mm]
\leq{}& C_S \cdot |B_h(x,2)|^{-2/n}\cdot  \int_{M}|\tinabla(\eta^{2}Q^{\frac{n}{2}})|^{2}d\vol_{h} \\
\leq {} & C_7\cdot|B_h(x,2)|^{-2/n}\cdot\left(\int_{M}\eta^{4}Q^{n-2}|\tinabla Q|^{2}d\vol_{h}
+\int_{M}\eta^{2}Q^{n}d\vol_{h}\right).
\end{split}
\end{equation}
Combining \eqref{W 1 n estimate eqn 3} and \eqref{W 1 n estimate eqn 4}, we conclude that
\begin{equation}\label{energy-ineq}
\begin{split}
& \frac{2}{n}\cdot\frac{\de}{\de t}\left(\int_{M}\eta^{4}Q^{n}d\vol_{h}\right)
+\frac{|B_h(x,2)|^{2/n}}{2C_7\Lambda}\left(\int_{M}(\eta^{4}Q^{n})^{\frac{n}{n-2}}d\vol_{h}\right)^{\frac{n-2}{n}} \\[1.5mm]
\leq {} & C_8\int_{M}(\eta^{4}Q^{n+2}+\eta^{2}Q^{n}+\eta^{4}Q^{n}+\eta^{4})d\vol_{h} \\
\leq {} & C_9\left(\int_{B_{h}(x,2)}Q^{n}d\vol_{h}\right)^{\frac{2}{n}}\left(\int_{M}(\eta^{4}Q^{n})^{\frac{n}{n-2}}d\vol_{h}\right)^{\frac{n-2}{n}}
+C_9\int_{B_{h}(x,2)}(Q^{n}+1)d\vol_{h}.
\end{split}
\end{equation}
Now we are going to control the first term on the right hand side by using assumption (ii). Let $\{z_i\}_{i=1}^N$ be a maximal set in $B_h(x,2)$ so that $d_h(z_i,z_j)>\frac14$ and
\begin{equation}
B_h(x,2)\subset \bigcup_{i=1}^N B_h(z_i,1)\subset B_h(x,4).
\end{equation}
Then the standard volume comparison argument shows that $N\leq C_n$ as $|\Rm(h)|\leq 1$.
Together with the definition of $Q$ and $\sigma\in(0,\delta_{1})$, we conclude that
\begin{equation}
\begin{split}
\int_{B_{h}(x,2)}Q^{n}d\vol_{h} &\leq \sum_{i=1}^N \int_{B_h(z_i,1)} |\tilde \nabla g(t)|^n d\mathrm{vol}_h+C_n\cdot \sigma^{n}\cdot|B_h(x,4)|\\[1mm]
&\leq C_n \cdot\delta_{1}^n \cdot |B_h(x,4)|\\[2mm]
&\leq \hat C_n\cdot\delta_{1}^n \cdot |B_h(x,2)|.
\end{split}
\end{equation}
Here we have used the volume comparison theorem in the last inequality. By putting this into \eqref{energy-ineq}, if we choose $\delta_{1}$ such that $C_{9}\hat{C}_{n}^{2/n}\delta_{1}^{2}\leq (2C_{7}\Lambda)^{-1}$, then
\begin{equation}
\frac{\de}{\de t}\left(\int_{M}\eta^{4}Q^{n}d\vol_{h}\right) \leq C_{10} |B_h(x,2)|.
\end{equation}
After integrating both sides on $[0,t]$ and letting $\sigma\to 0$, we obtain
\begin{equation}
\int_{B_{h}(x,1)}|\tinabla g(t)|^{n}d\vol_{h}
\leq \int_{B_{h}(x,2)}|\tinabla g(0)|^{n}d\vol_{h}+C_{10}|B_h(x,2)|\cdot t
\end{equation}
for all $(x,t)\in M\times[0,T]$.  This completes the proof by applying the volume comparison theorem again.
\end{proof}

Next, we consider the trace estimate.
\begin{lma}\label{borderline-trace}
For any $\Lambda_0>1$, there is $\delta_2(n,\Lambda_0),S(n,\Lambda_0)>0$ such that the following holds. Suppose $g(t)$ is a smooth solution to the Ricci-Deturck $h$-flow on $M\times [0,T]$ for some background metric $h$ satisfying $(\star)$ so that
\begin{enumerate}
\item[(i)] $\Lambda_0^{-1} h\leq g_0\leq \Lambda_0 h$ on $M$;
\item[(ii)] $\forall x\in M$ and $t\in (0,T]$,
$$\left(\fint_{B_h(x,1)} |\tilde\nabla g(t)|^n d\mathrm{vol}_h\right)^{1/n} \leq \delta_2;$$
\item[(iii)] $\sup_{M\times [0,T]} |\tilde \nabla^k g(t)|<+\infty$ for all $k\in \mathbb{N}$.
\end{enumerate}
then on $M\times [0,T\wedge S]$,
\begin{equation}
\Lambda^{-1}h\leq g(t)\leq \Lambda h
\end{equation}
where $\Lambda=\Upsilon_n\Lambda_0$ for some large dimensional constant $\Upsilon_n>0$.
\end{lma}

\begin{proof}
We assume $T\leq 1$.  By assumptions (i) and (iii),  we may find $T_1>0$ such that for all $t\in [0,T_1)$, the conclusion holds. If $T_1<1$, then the conclusion fails at $t=T_1$.  It suffices to show that if $\delta_{2}(n,\Lambda_0)$ is sufficiently small, then $T_1$ is bounded from below by constant depending only on $n$ and $\Lambda_0$.   We will use $C_i$ to denote constants depending only on $n$ and $\Lambda_0$.

For any $x\in M$ and $t\in (0,T_1]$, then assumption (ii) and the volume comparison theorem imply
\begin{equation}
\begin{split}
t^{n/2} \fint_{B_h(x,\sqrt{t})} |\tilde \nabla g(t)|^n d\mathrm{vol}_h
&=\frac{t^{n/2}}{|B_h(x,\sqrt{t})|} \int_{B_h(x,\sqrt{t})} |\tilde \nabla g(t)|^n d\mathrm{vol}_h \\
&\leq t^{n/2}\cdot \frac{|B_h(x,1)|}{|B_h(x,\sqrt{t})|} \fint_{B_h(x,1)} |\tilde \nabla g(t)|^n d\mathrm{vol}_h \\
&\leq  t^{n/2} \cdot \frac{V_{-1}(1)}{V_{-1}(\sqrt{t})}\cdot \delta_{2}^{n}\\[1.5mm]
&\leq C_n\delta_{2}^{n} .
\end{split}
\end{equation}
where $V_{-1}(r)$ denotes the volume of geodesic ball with radius $r$ in the space form $M_{-1}^{n}$. Since on $M\times [0,T_1]$,  we have $\Lambda^{-1}h\leq g(t)\leq \Lambda h$. By Lemma~\ref{lma:localEstimate-fromC^1conc}, for all $t\in (0,T_1]$,
\begin{equation}\label{pre-C1}
\sup_M |\tilde \nabla g(t)|\leq \frac{C_0\delta_{2}^{1/2}}{\sqrt{t}}
\end{equation}
for some constant $C_0(n,\Lambda_0)>0$ as long as $\delta_{2}$ is sufficiently small depending only on $n$ and $\Lambda_0$.

We want to consider the upper bound of the metric $g(t)$ at $t=T_1$. It suffices to control the upper bound of $\tr_hg$.  Since the conclusion fails at $t=T_1$, there is $x_0\in M$ such that $\tr_hg (x_0,T_1)>\frac12\Lambda$.  By \eqref{pre-C1}, for all $z\in B_h(x_0,r)$,
\begin{equation}
\tr_hg (x,T_1)\geq \tr_hg(x_0,T_1)-C_1\delta_{2}^{1/2} T_1^{-1/2}r
\end{equation}
and hence
\begin{equation}\label{inter-trace-1}
\begin{split}
\left(\int_{B_h(x_0,r)} \tr_hg\; d\mathrm{vol}_h\right)\bigg|_{t=T_1}
\geq \left( \frac12 \Lambda-C_1\delta_{2}^{1/2} T_1^{-1/2}r\right)|B_h(x_0,r)|.
\end{split}
\end{equation}
When $r$ is small enough so that the right hand side is positive. We may assume $r\leq \frac12$ and fix the choice of $r$ later.

On the other hand, taking trace on \eqref{eqn: h-flow} infers that
\begin{equation}\label{equ-trace}
\frac{\de}{\de t}\tr_{h}g
\leq g^{ab}\tinabla_{a}\tinabla_{b}\tr_hg+C_2(|\tinabla g|^{2}+1).
\end{equation}
Let $\eta$ be a cut-off function on $M$ such that $\eta \equiv 1$ on $B_{h}(x_{0},r)$, $\eta\equiv0$ on $M\setminus B_{h}(x_{0},2r)$ and $|\tinabla\eta|\leq 10^4r^{-1}$. Integrating \eqref{equ-trace} and using integration by parts yield that
\begin{equation}
\begin{split}
& \frac{\de}{\de t}\int_{B_{h}(x_{0},2r)} \eta\tr_{h}g(t)\; d\vol_{h} \\
\leq {} & C_3\int_{B_{h}(x_{0},2r)}|\tinabla \eta|\cdot|\tinabla g(t)|\;d\vol_{h}
+C_2\int_{B_{h}(x_{0},2r)}(|\tinabla g(t)|^{2}+1)d\vol_{h} \\
\leq {} & C_3\left(\int_{B_{h}(x_{0},2r)}|\tinabla\eta|^{\frac{n}{n-1}}d\vol_{h}\right)^{\frac{n-1}{n}}
\left(\int_{B_{h}(x_{0},2r)}|\tinabla g(t)|^{n}d\vol_{h}\right)^{\frac{1}{n}} \\
& +C_2|B_h(x_0,2r)|^{\frac{n-2}{n}}
\left(\int_{B_{h}(x_{0},2r)}|\tinabla g(t)|^{n}d\vol_{h}\right)^{\frac{2}{n}}+C_2|B_h(x_0,2r)|\\
\leq {} &  C_3r^{-1}\delta_{2}|B_h(x_0,2r)|^\frac{n-1}{n}|B_h(x_0,1)|^\frac1n
+C_2\delta_{2}^2 |B_h(x_0,2r)|^{\frac{n-2}{n}} |B_h(x_0,1)|^\frac{2}{n}\\[2.5mm]
&+C_2|B_h(x_0,2r)|\\[2mm]
\leq {} &C_4 (\delta_{2} r^{-2}+1) |B_h(x_0,2r)|.
\end{split}
\end{equation}
Here we have used the volume comparison theorem in the last inequality. By integrating the above on $[0,T_1]$ together with (i), we conclude that
\begin{equation}
\begin{split}
&\quad \left(\int_{B_h(x_0,r)} \tr_hg\; d\mathrm{vol}_h\right)\bigg|_{t=T_1}\\
&\leq \int_{B_h(x_0,2r)} \tr_hg_0\; d\mathrm{vol}_h+ C_4 T_1(\delta_{2} r^{-2}+1) |B_h(x_0,2r)|\\[2mm]
&\leq   |B_h(x_0,2r)| \left( n\Lambda_0+C_4 T_1+C_4 T_1\delta_{2} r^{-2}  \right).
\end{split}
\end{equation}
Combining this with \eqref{inter-trace-1} and the volume comparison theorem,
\begin{equation}
\Upsilon_n \Lambda_0-C_1\delta_{2}^{1/2} T_1^{-1/2}r \leq  C_n \left( n\Lambda_0+C_4 T_1+C_4 T_1\delta_{2} r^{-2} \right).
\end{equation}
If we fix $\Upsilon_n=2nC_n$ and choose $r=T_1^{1/2}$, then we see that $T_1\geq C_5^{-1}$ if $\delta_{2}$ is small enough depending only on $n$ and $\Lambda_0$. This completes the proof on the upper bound of $g(t)$.

The lower bound of $g(t)$ can be proved using an identical argument by considering the function $\tr_{g(t)}h$ instead.  The evolution equation of $\tr_{g(t)}h$ is almost identical to that of \eqref{equ-trace}, see \cite{Shi1989,Simon2002} for example.
\end{proof}

The following proposition shows that the solution of $h$-flow will be uniformly smooth up to $t=0$ on region where $g_0$ is smooth.  This will be useful in studying the regularity of singular metric which is smooth outside some singular set.
\begin{prop}\label{prop:local-smoothness-Ck}
Suppose $g(t)$ is a smooth solution to the Ricci-Deturck $h$-flow on $M\times [0,T]$ for some background metric $h$ satisfying $(\star)$ so that
\begin{enumerate}
\item[(i)] $\Lambda^{-1}h\leq g(t)\leq \Lambda h$ on $M\times [0,T]$;
\item[(ii)] For all $k\in \mathbb{N}$, there is $C(n,k,\Lambda)>0$ such that for all $t\in (0,T]$,
$$\sup_M |\tilde\nabla^k g(t)|\leq C(n,k,\Lambda)t^{-k/2};$$
\item[(iii)] There is $\Omega\Subset M, k_0\in \mathbb{N}$ so that $\sum_{m=1}^{k_0}\sup_\Omega|\tilde\nabla^mg_0| <\Upsilon_0$;
\end{enumerate}
then for all $\tilde\Omega\Subset \Omega$,  there is $L(n,h,\tilde\Omega,\Omega,\Lambda,\Upsilon_0)$ such that for all $t\in [0,T]$,
\begin{equation}
\sup_{\tilde\Omega} \sum_{m=1}^{k_0}|\tilde\nabla^{m} g(t)|\leq L.
\end{equation}
\end{prop}
\begin{proof}
We will use $C_i$ to denote constants depending only on $n$, $\Lambda$ and $\Upsilon_0$. By the proof of \cite[Lemma 4.2]{Shi1989}, it suffices to consider the case of $k_0=1$. Let $x_0\in \tilde\Omega$ and $r_0\in(0,1)$ such that $B_h(x_0,r_0)\Subset \Omega$. We claim that there is $L(n,h,\Omega,\Lambda,\Upsilon_0)>0$ such that
\begin{equation}\label{prop:local-smoothness-Ck claim}
|\tilde\nabla g(x,t)| < L\cdot \rho(x)^{-1}
\end{equation}
for all $(x,t)\in B_h(x_0,r_0)\times [0,T]$ where
\begin{equation}
\rho(x)=\inf\{ r\in (0,r_0): B_h(x,r)\Subset B_h(x_0,r_0)\}.
\end{equation}
Then the conclusion follows immediately from the above claim \eqref{prop:local-smoothness-Ck claim}.

Next, we will prove \eqref{prop:local-smoothness-Ck claim}. The large constant $L$ will be determined later. We assume $L>\Upsilon_0$. Then by $r_{0}<1$ and assumption (iii), there is $T_{1}>0$ such that \eqref{prop:local-smoothness-Ck claim} holds on $B_h(x_0,r_0)\times [0,T_{1}]$. If $T_{1}=T$, then we are done. If $T_{1}<T$, then
\begin{equation}\label{prop:local-smoothness-Ck eqn 1}
|\tilde\nabla g(x,t)| < L\cdot \rho(x)^{-1}
\end{equation}
for all $(x,t)\in B_h(x_0,r_0)\times [0,T_{1})$, while there is $x_{1}\in B_{h}(x_{0},r_{0})$ such that
\begin{equation}\label{prop:local-smoothness-Ck eqn 2}
|\tilde\nabla g(x_{1},T_{1})| = L\cdot \rho_{1}^{-1},
\end{equation}
where $\rho_{1}=\rho(x_{1})$. \eqref{prop:local-smoothness-Ck eqn 1} implies that for all $(x,t)\in B_h(x_1,\frac12\rho_1)\times [0,T_1]$,
\begin{equation}
|\tilde\nabla g(x,t)| \leq 2L\rho_{1}^{-1}.
\end{equation}
Combining assumption (i) and \eqref{prop:local-smoothness-Ck eqn 2}, we obtain $L\rho_{1}^{-1}\leq C_{0}T_{1}^{1/2}$.

To better illustrate the dependency,  we let
\begin{equation}
r=\frac18 \rho_1, \ \ \tilde T_1=64\rho_1^{-2}T_1
\end{equation}
and consider the rescaled metrics $\hat g(t)=r^{-2}g(r^2t)$ and $\hat h=r^{-2}h$ on $B_{\hat h}(x_1,4)\times [0,\tilde T_1]$ which satisfy
\begin{equation}\label{rescaled-conditions}
\left\{
\begin{array}{ll}
\Lambda^{-1} \hat h\leq \hat g(t)\leq \Lambda \hat h;\\[1mm]
|\nabla^{\hat h}\hat g(t)|\leq \frac14 L;\\[1mm]
|\nabla^{\hat h}\hat g(0)|\leq \frac18 \rho_1\Upsilon_0 ;\\[1mm]
|\nabla^{\hat h}\hat g|(x_1,\tilde T_1)=\frac18 L.\\
\end{array}
\right.
\end{equation}
We will work on $B_{\hat h}(x_1,4)\times [0,\tilde T_1]$. Recalling $L\rho_{1}^{-1}\leq C_{0}T_{1}^{-1/2}$, we obtain
\begin{equation}\label{prop:local-smoothness-Ck eqn 3}
L^2\tilde T_1 = 64L^{2}\rho_{1}^{-2}T_{1} \leq  64 C_{0}^{2}.
\end{equation}
Using \eqref{eqn: h-flow} and \eqref{rescaled-conditions}, direct calculation shows that there are $C_1,C_2,C_3>0$ such that
\begin{equation}
\left\{
\begin{array}{ll}
\left( \frac{\partial}{\partial t}-\hat g^{ij}\nabla^{\hat h}_i\nabla^{\hat h}_h\right)\left(e^{-C_1L^2 t}\tr_{\hat h}\hat g\right)
\leq -e^{-C_1L^2 t}|\nabla^{\hat h}\hat g|^2;\\[3mm]
\left( \frac{\partial}{\partial t}-\hat g^{ij}\nabla^{\hat h}_i\nabla^{\hat h}_h\right)  \left( e^{-C_2L^2t}|\nabla^{\hat h}\hat g|^2 \right)
\leq -C_3^{-1}e^{-C_2L^2t}|\nabla^{\hat h,2}\hat g|^2.
\end{array}
\right.
\end{equation}

Let $\eta$ be a cut-off function on $M$ such that $\eta\equiv 1$ on $B_{\hat h}(x_1,1)$, $\eta\equiv 0$ on $M\setminus B_{\hat h}(x_1,4)$ and $|\nabla^{\hat h}\eta|\leq10^{4}$. Consider the test function
\begin{equation}
F=e^{-\tilde CL^2t} \eta^{2} |\nabla^{\hat h}\hat g|^2+\tilde C e^{-C_1L^2 t} \tr_{\hat h}\hat g
\end{equation}
where $\tilde C$ is a constant to be determined later. \eqref{rescaled-conditions} and the Cauchy-Schwarz inequality implies
\begin{equation}
\begin{split}
& \left( \frac{\partial}{\partial t}-\hat g^{ij}\nabla^{\hat h}_i\nabla^{\hat h}_h\right)F \\[1mm]
\leq {} &  10^{4}e^{-\tilde CL^2t}\eta\big|\nabla^{\hat h}|\nabla^{\hat h}\hat g|^2\big|
-C_3^{-1}e^{-\tilde CL^2t}\eta^{2}|\nabla^{\hat h,2}\hat g|^2
-\tilde C e^{-C_1L^2t} |\nabla^{\hat h}\hat g|^2\\[1mm]
\leq {} & C_4 e^{-\tilde C L^2 t}|\nabla^{\hat h}\hat g|^2-\tilde Ce^{-C_1L^2t} |\nabla^{\hat h}\hat g|^2.
\end{split}
\end{equation}
By \eqref{prop:local-smoothness-Ck eqn 3}, we have $L^2t\leq L^{2}T_{1}\leq64C_{0}^{2}$. Then after choosing $\tilde C$ sufficiently large depending only on $n$ and $\Lambda$,
\begin{equation}
\left( \frac{\partial}{\partial t}-\hat g^{ij}\nabla^{\hat h}_i\nabla^{\hat h}_h\right)F \leq 0.
\end{equation}
Therefore, the maximum principle implies
\begin{equation}
F(x_1,\tilde T_1)
\leq \max\left\{\sup_{B_{\hat h}(x_1,4)} F(\cdot, 0), \sup_{\partial B_{\hat h}(x_1,4)\times [0,\tilde T_1]} F \right\}
\leq  \frac18 \rho_1 \Upsilon_0+n\tilde C \Lambda.
\end{equation}
On the other hand, by \eqref{rescaled-conditions} and \eqref{prop:local-smoothness-Ck eqn 3},
\begin{equation}
F(x_1,\tilde T_1)\geq \frac1{64}e^{-64 C_0^{2}\tilde C} L^2.
\end{equation}
However, the above is impossible if we initially choose $L$ to be sufficiently large. This shows that with this choice of $L$, we must have $T_{1}=T$, which proves the claim \eqref{prop:local-smoothness-Ck claim}. Then we complete the proof of case $k_0=1$. When $k_{0}>1$, the higher order estimate can be proved by using a similar argument or using Bernstein-Shi's trick as in \cite{Shi1989}.
\end{proof}

We remark here that the above estimate we obtained is far from being optimal but it suffices for our purpose.  In particular, using iteration method one should be able to improve dependency of some constants.  Next, we will establish a uniform time zero regularity of the Ricci-Deturck flow with respect to the $L^p_{\loc}$  topology under metric equivalence and small gradient concentration assumptions.

\begin{prop}\label{timezero-regularity-W1n}
For any $\Lambda>1$, there is $\delta_{3}(n,\Lambda)>0$ such that the following holds. Suppose $g(t)$ is a smooth solution to the Ricci-Deturck $h$-flow on $M\times [0,T]$ for some background metric $h$ satisfying $(\star)$ so that
\begin{enumerate}
\item[(i)] $\Lambda^{-1} h\leq g(t)\leq \Lambda h$ on $M\times [0,T]$;
\item[(ii)] $\forall x\in M$ and $t\in (0,T]$,
$$\left(\fint_{B_h(x,1)} |\tilde\nabla g(t)|^n d\mathrm{vol}_h\right)^{1/n} \leq \delta_{3};$$
\end{enumerate}
then for all $p>0$ there is $L_p(p,n,\Lambda)>0$ such that for $x\in M$ and $t,s\in [0,T]$,
\begin{equation}
\fint_{B_h(x,1)} | g(t)-g(s)|^p d\mathrm{vol}_h \leq L_p|t-s|.
\end{equation}
\end{prop}

\begin{proof}
We use $C_i$ to denote constants depending only on $n$ and $\Lambda$. For $x\in M$. Let $\eta$ be a smooth cut-off function on $M$ such that $\eta\equiv 1$ on $B_h(x,1)$, $\eta\equiv 0$ on $M\setminus B_h(x,2)$ and $|\tinabla\eta|\leq10^{4}$. We assume without loss of generality that $t>s$ and define $\alpha=g(t)-g(s)$. Then $\de_{t}\alpha_{ij}=\de_{t}g_{ij}$ and assumption (i) implies $|\alpha|\leq2\Lambda$. Using \eqref{eqn: h-flow} and integrating by parts,
\begin{equation}\label{timezero-regularity-W1n eqn 1}
\begin{split}
& \frac{\partial}{\partial t}\int_M \eta^{2} |\a|^2 \; d\mathrm{vol}_h \\
= {} & 2\int_M \eta^{2} \, \langle\a, g^{pq}\tilde\nabla_p \tilde\nabla_q g+ \widetilde{\mathrm{Rm}}*g^{-1}*g
+ g^{-1}*g^{-1}*\tilde\nabla g*\tilde \nabla g \rangle  \; d\mathrm{vol}_h \\
\leq {} & 2\int_M \eta^{2} h^{ik}h^{jl} \a_{kl}g^{pq}\tilde\nabla_p \tilde\nabla_q g_{ij}d\vol_{h}
+C_{1}\int_{M}\eta^{2}|\alpha|(1+|\tinabla g|^{2})d\vol_{h} \\
= {} & -2\int_M h^{ik}h^{jl}\tinabla_{p}(\eta^{2}\a_{kl}g^{pq})\tilde\nabla_q g_{ij}d\vol_{h}
+2C_{1}\Lambda\int_{M}\eta^{2}(1+|\tinabla g|^{2})d\vol_{h}.
\end{split}
\end{equation}
For the first term on the right hand side, using $g_{ij}=\alpha_{ij}+g_{ij}(s)$, the Cauchy-Schwarz inequality and $|\alpha|\leq 2\Lambda$,
\begin{equation}
\begin{split}
& -2\int_M h^{ik}h^{jl}\tinabla_{p}(\eta^{2}\a_{kl}g^{pq})\tilde\nabla_q g_{ij}d\vol_{h} \\
= {} & -2\int_M h^{ik}h^{jl}\tinabla_{p}(\eta^{2}\a_{kl}g^{pq})\tilde\nabla_q\alpha_{ij}d\vol_{h}
-2\int_M h^{ik}h^{jl}\tinabla_{p}(\eta^{2}\a_{kl}g^{pq})\tilde\nabla_q g_{ij}(s)d\vol_{h} \\
\leq {} & -2\Lambda^{-1}\int_{M}\eta^{2}|\tinabla\alpha|^{2}d\vol_{h}
+C_{2}\int_{M}(\eta|\alpha||\tinabla\alpha|+\eta^{2}|\alpha||\tinabla g||\tinabla\alpha|)d\vol_{h} \\
& +C_{2}\int_M (\eta|\alpha|+\eta^{2}|\tinabla\alpha|+\eta^{2}|\alpha||\tinabla g|)|\tinabla g(s)|d\vol_{h} \\
\leq {} & C_{2}\int_{B_{h}(x,2)}(|\alpha|^{2}+|\alpha|^{2}|\tinabla g|^{2}
+|\alpha||\tinabla g(s)|+|\tinabla g(s)|^{2}+ |\alpha||\tinabla g||\tinabla g(s)|)d\vol_{h} \\
\leq {} & C_{2}\Lambda^{2}\int_{B_{h}(x,2)}(1+|\tinabla g|^{2}+|\tinabla g(s)|^{2})d\vol_{h}.
\end{split}
\end{equation}
Substituting this into \eqref{timezero-regularity-W1n eqn 1},
\begin{equation}
\frac{\partial}{\partial t}\int_M \eta^{2}|\a|^2 \; d\mathrm{vol}_h
\leq C_{3}\int_{B_{h}(x,2)}(1+|\tinabla g|^{2}+|\tinabla g(s)|^{2})d\vol_{h}.
\end{equation}
Using assumption (ii), the volume comparison theorem and a covering argument, we obtain
\begin{equation}
\begin{split}
\int_{B_{h}(x,2)}(1+|\tinabla g|^{2}+|\tinabla g(s)|^{2})d\vol_{h}
\leq C_{4}(1+\delta_{3}^{2})|B_{h}(x,2)| \leq C_{5}|B_{h}(x,1)|.
\end{split}
\end{equation}
It then follows that
\begin{equation}
\frac{\partial}{\partial t}\int_M \eta^{2}|\a|^2 \; d\mathrm{vol}_h
\leq C_{6}|B_{h}(x,1)|.
\end{equation}
Integrating on $[s,t]$ and recalling $\alpha=g(t)-g(s)$, we obtain
\begin{equation}
\fint_{B_{h}(x,1)} |g(t)-g(s)|^{2} \; d\mathrm{vol}_h
\leq \frac{1}{|B_{h}(x,1)|}\int_M \eta^{2}|\a|^2 \; d\mathrm{vol}_h
\leq C_{6}(t-s).
\end{equation}
This completes the proof of the case $p=2$. When $0<p<2$, the conclusion follows from the H\"older inequality. When $p>2$, using assumption (i), we obtain
\begin{equation}
\fint_{B_{h}(x,1)} |g(t)-g(s)|^{p} \; d\mathrm{vol}_h
\leq \Lambda^{p-2} \fint_{B_{h}(x,1)} |g(t)-g(s)|^{2} \; d\mathrm{vol}_h
\leq C_{6}\Lambda^{p-2}(t-s).
\end{equation}
\end{proof}

We end this section by showing a $W^{1,n}_{\loc}$ stability along flow with small gradient concentration.
\begin{prop}\label{W1n stability}
For any $\Lambda>1$, there is $\delta_{4}(n,\Lambda)>0$ such that the following holds.
Suppose $g(t)$ is a smooth solution to the Ricci-Deturck $h$-flow on $M\times [0,T]$ for some background metric $h$ satisfying $(\star)$ so that
\begin{enumerate}
\item[(i)] $\Lambda^{-1} h\leq g(t)\leq \Lambda h$ on $M\times [0,T]$;
\item[(ii)] $\forall x\in M$ and $t\in (0,T]$,
\begin{equation*}
\left(\fint_{B_h(x,1)} |\tilde\nabla g(t)|^n d\mathrm{vol}_h\right)^{1/n} \leq \delta_{4}.
\end{equation*}
\end{enumerate}
Then for any smooth metric $\bar{g}$ on $M$, there is $\ov{L}(n,\Lambda,\bar{g})>0$ such that for all $x\in M$ and $t,s\in [0,T]$,
\begin{equation}
\int_{B_h(x,1)} |\tilde\nabla(g(t)-\bar{g})|^n d\mathrm{vol}_h \leq \int_{B_h(x,2)} |\tilde\nabla(g(0)-\bar{g})|^n d\mathrm{vol}_h+\ov{L}|B_h(x,2)| \cdot t.
\end{equation}
\end{prop}

\begin{proof}
We will use $C_{i}$ to denote constants depending only on $n$ and $\Lambda$, while $\ov{C}_{i}$ to denote constants depending only on $n$, $\Lambda$ and $\bar{g}$. Define
\begin{equation}
\beta_{ij} = g_{ij}(t) - \bar{g}_{ij}.
\end{equation}
Then along the Ricci-Deturck $h$-flow, we have
\begin{equation}
\frac{\de}{\de t}\beta_{ij} = g^{ab}\tinabla_{a}\tinabla_{b}\beta_{ij}+g^{-1}*\tinabla^{2}\bar{g}+\widetilde{\mathrm{Rm}}*g^{-1}*g+g^{-1}*g^{-1}*\tinabla g*\tinabla g
\end{equation}
and so
\begin{equation}
\begin{split}
& \frac{\de}{\de t}|\tinabla\beta|^{2}+2\Lambda^{-1}|\tinabla^{2}\beta|^{2} \\
\leq {} & g^{ab}\tinabla_{a}\tinabla_{b}|\tinabla\beta|^{2}+C_{1}|\tinabla\beta||\tinabla g||\tinabla^{2}\beta|+C_{1}|\tinabla\beta|^{2} \\
&+\ov{C}_{1}(|\tinabla g|+1)|\tinabla\beta|
+C_{1}|\tinabla\beta|\left(|\tinabla g||\tinabla^{2}g|+|\tinabla g|^{3} \right).
\end{split}
\end{equation}
Using $g=\beta+\bar{g}$ and the Cauchy-Schwarz inequality,
\begin{equation}
\begin{split}
& \frac{\de}{\de t}|\tinabla\beta|^{2}+2\Lambda^{-1}|\tinabla^{2}\beta|^{2} \\
\leq {} & g^{ab}\tinabla_{a}\tinabla_{b}|\tinabla\beta|^{2}+C_{2}|\tinabla\beta||\tinabla g||\tinabla^{2}\beta|
+C_{2}|\tinabla g|^{2}|\tinabla\beta|^{2}+C_{2}|\tinabla\beta|^{2}+\ov{C}_{2} \\
&+C_{2}|\tinabla\beta|\left(|\tinabla g||\tinabla^{2}\beta|+|\tinabla g||\tinabla^{2}\bar g|+|\tinabla g|^{2}|\tinabla\beta|
+|\tinabla g|^{2}|\tinabla\bar g| \right) \\
\leq {} & g^{ab}\tinabla_{a}\tinabla_{b}|\tinabla\beta|^{2}+C_{3}|\tinabla\beta||\tinabla g||\tinabla^{2}\beta|
+C_{3}|\tinabla g|^{2}|\tinabla\beta|^{2}+\ov{C}_{3}|\tinabla\beta|^{2}+\ov{C}_{3} \\
\leq {} & g^{ab}\tinabla_{a}\tinabla_{b}|\tinabla\beta|^{2}+\Lambda^{-1}|\tinabla^{2}\beta|^{2}
+C_{3}|\tinabla g|^{2}|\tinabla\beta|^{2}+\ov{C}_{3}|\tinabla\beta|^{2}+\ov{C}_{3}.
\end{split}
\end{equation}
For $\sigma\in(0,\delta_{4})$, we define the smooth function $H$ by
\begin{equation}
H = \sqrt{|\tinabla\beta|^{2}+\sigma^{2}}
\end{equation}
Then we obtain
\begin{equation}
\frac{\de}{\de t}H^{2}+\Lambda^{-1}|\tinabla\beta|^{2} \leq g^{ab}\tinabla_{a}\tinabla_{b}H^{2}+C_{4}|\tinabla g|^{2}H^{2}+\ov{C}_{4}H^{2}+\ov{C}_{4}.
\end{equation}
Let $\eta$ be a cut-off function such that $\eta\equiv1$ on $B_{h}(x,1)$, $\eta\equiv0$ on $M\setminus B_{h}(x,2)$ and $|\tilde\nabla \eta|\leq 10^4$ on $M$. By the similar calculation of \eqref{energy-ineq},
\begin{equation}\label{W 1 n estimate eqn 5}
\begin{split}
& \frac{2}{n}\cdot\frac{\de}{\de t}\left(\int_{M}\eta^{4}H^{n}d\vol_{h}\right)
+\frac{|B_{h}(x,2)|^{2/n}}{C_{5}\Lambda}\left(\int_{M}(\eta^{4}H^{n})^{\frac{n}{n-2}}d\vol_{h}\right)^{\frac{n-2}{n}} \\[1.5mm]
\leq {} & C_{5}\int_{M}\eta^{4}|\tinabla g|^{2}H^{n}d\vol_{h}
+\ov{C}_{5}\int_{M}(\eta^{2}H^{n}+\eta^{4}H^{n}+\eta^{4})d\vol_{h}. \\
\leq {} & C_{5}\left(\int_{B_{h}(x,2)}|\tinabla g|^{n}d\vol_{h}\right)^{\frac{2}{n}}\left(\int_{M}(\eta^{4}H^{n})^{\frac{n}{n-2}}d\vol_{h}\right)^{\frac{n-2}{n}}
+\ov{C}_{5}\int_{B_{h}(x,2)}(H^{n}+1)d\vol_{h}.
\end{split}
\end{equation}
Using assumption (ii), we obtain
\begin{equation}\label{W 1 n estimate eqn 6}
\left(\int_{B_{h}(x,2)}|\tinabla g|^{n}d\vol_{h}\right)^{\frac{2}{n}} \leq C_{n}\delta_{4}^{2}|B_{h}(x,2)|^{2/n}
\end{equation}
and
\begin{equation}\label{W 1 n estimate eqn 7}
\int_{B_{h}(x,2)}(H^{n}+1)d\vol_{h} \leq \ov{C}_{6}|B_{h}(x,2)|.
\end{equation}
Choosing $\delta_{4}$ sufficiently small depending only on $n$ and $\Lambda$, then \eqref{W 1 n estimate eqn 5}, \eqref{W 1 n estimate eqn 6} and \eqref{W 1 n estimate eqn 7} imply
\begin{equation}
\frac{\de}{\de t}\left(\int_{M}\eta^{4}H^{n}d\vol_{h}\right)
\leq \ov{C}_{7}|B_{h}(x,2)|.
\end{equation}
After integrating both sides on $[0,t]$ and letting $\sigma\to0$, we obtain
\begin{equation}
\int_{B_h(x,1)} |\tilde\nabla(g(t)-\bar{g})|^n d\mathrm{vol}_h \leq \int_{B_h(x,2)} |\tilde\nabla(g(0)-\bar{g})|^n d\mathrm{vol}_h+\ov{C}_{8}|B_h(x,2)| \cdot t.
\end{equation}
\end{proof}

\section{Existence of Ricci-Deturck flow with rough initial data}\label{Sec:Existence}

In this section, we will establish the existence of the Ricci-Deturck $h$-flow starting from $g_0$ which is bi-Lipschitz on $M$ and is $W^{1,n}$. We start with the proof of Theorem~\ref{thm-existence} which is a quantitative version of the existence theory.

\begin{proof}[Proof of Theorem~\ref{thm-existence}]
We assume $(M,h)$ to be complete non-compact. The compact case can be proved using a similar but simpler argument.

By Proposition~\ref{prop:RF-smoothing}, we may assume that the background metric $h$ satisfies $(\star)$ by a slight perturbation. By the result of \cite{Tam2010} (see also \cite{GreenWu1979}) and $|\Rm(h)|\leq 1$, there is $\rho\in C_{\loc}^\infty(M)$ such that $|\tilde\nabla\rho|^2+|\tilde\nabla^2 \rho|\leq 1$ and
\begin{equation}
C_n^{-1}(d_h(\cdot,p)+1)\leq \rho(\cdot) \leq C_n(d_h(\cdot,p)+1)
\end{equation}
for some dimensional constant $C_n$ and fixed point $p\in M$.

Let $\phi$ be a smooth function on $[0,+\infty)$ such that $\phi\equiv 1$ on $[0,1]$, $\phi\equiv 0$ on $[2,+\infty)$ and $0\leq- \phi'\leq 10$.  Consider the following sequence of metrics:
\begin{equation}
h_{i}=\phi\left(i^{-1}\rho\right)g_0 +\left(1-\phi(i^{-1}\rho )\right)h
\end{equation}
so that $h_i\equiv h$ at infinity and coincides with $g_0$ on large compact set. Moreover, $h_i\to g_0$ in $W^{1,n}_{\loc}$ as $i\to +\infty$.  Since $h_i=h$ outside a large compact set, we may mollify $h_i$ slightly and obtain a smooth metric $g_{i,0}$ so that $g_{i,0}=h$ outside compact set and $g_{i,0}\to g_0$ in $W^{1,n}_{\loc}$ as $i\to+\infty$.  More precisely,  we may assume that for all $i\in\mathbb{N}$ sufficiently large,
\begin{equation}
\left\{
\begin{array}{ll}
\Lambda_0^{-1} h\leq g_{i,0}\leq \Lambda_0 h;\\
\displaystyle \left(\fint_{B_h(x,1)} |\tilde \nabla g_{i,0}|^n \; d\mathrm{vol}_h \right)^{1/n}\leq \e
\end{array}
\right.
\end{equation}
for all $x\in M$ and for some $\e<\e_0$. We will specify the choice of $\e_0$ later.

\medskip

By \cite[Theorem A.1]{LammSimon2021} which is a slight modification of Shi's classical existence theory \cite{Shi1989}, there is a short-time solution to the Ricci-Deturck $h$-flow $g_i(t)$ on $M\times [0,S_i]$ for some $S_i>0$ such that $g_i(0)=g_{i,0}$ and $\sup_M |\tilde \nabla^m g_i(t)|<+\infty$ for all $m\in \mathbb{N}$ on $M\times [0,S_i]$.  Here we remark that the injectivity radius lower bound assumption in  \cite[Theorem A.1]{LammSimon2021} is unnecessary as $h$ has bounded geometry. Alternatively, we can apply the existence result in \cite{Shi1989} and obtain a Ricci flow from $g_{i,0}$ and pull-back to a Ricci-Deturck $h$-flow as all metrics considered are smooth with bounded curvature.

We may assume $S_i$ is the maximal existence time of the flow $g_{i}(t)$. Let $T_i\in (0,S_{i})$ be the maximal time such that for all $t\in [0,T_i)$,
\begin{enumerate}
\item[(i')] $(2\Lambda)^{-1} h< g_i(t)< 2\Lambda h$ on $M\times [0,T_i)$;
\item[(ii')] $\forall x\in M$ and $t\in (0,T_i]$,
$$\left(\fint_{B_h(x,1)} |\tilde\nabla g_i(t)|^n d\mathrm{vol}_h\right)^{1/n} < \delta;$$
\item[(iii')] for all $m\in \mathbb{N}$, $\sup_M |\tilde \nabla^m g_i(t)|<+\infty$ on $[0,T_i)$
\end{enumerate}
where $\Lambda=\Upsilon_n\Lambda_0$ is obtained in Lemma~\ref{borderline-trace} and $\delta(n,\Lambda_0)=\min\{ \delta_1(n,\Lambda),\delta_2(n,\Lambda_0)\}$ is obtained in Lemma~\ref{borderline-C1} and  Lemma~\ref{borderline-trace} using this choice of $\Lambda$.

We will show that $S_i$ is bounded from below uniformly depending only on $n$ and $\Lambda_{0}$. It suffices to estimate the lower bound of $T_i$. By Lemma~\ref{lma:localEstimate-fromC^1} and Lemma~\ref{lma:localEstimate-fromC^1conc}, (iii') holds up to $t=T_i$. Therefore the maximality of $T_{i}$ implies that either (i') or (ii') fails at $t=T_i$. Using Lemma~\ref{borderline-C1}, a covering argument and the volume comparison theorem, we deduce that
\begin{equation}
\begin{split}
\int_{B_h(x,1)} |\tilde\nabla g_i(t)|^n d\mathrm{vol}_h &\leq \int_{B_h(x,2)} |\tilde\nabla g_{i,0}|^n d\mathrm{vol}_h+L_1|B_h(x,2)| \cdot t\\[1mm]
&\leq  (C_n\e^n+C_{n}L_1t)\cdot |B_h(x,1)|.
\end{split}
\end{equation}
Therefore,  by Lemma~\ref{borderline-trace}, if $\e_0(n,\Lambda_0)$ is sufficiently small, then $T_i> S(n,\Lambda_0)$.

Restricting $g_i(t)$ on $M\times [0,S]$, we may apply Lemma~\ref{lma:localEstimate-fromC^1} and Lemma~\ref{lma:localEstimate-fromC^1conc} to show that  $g_i(t)$ is uniformly $C^k_{\loc}$ bounded for any $[a,b]\subset (0,S]$ and uniformly bi-Lipschitz up on $[0,S]$. By diagonal subsequent argument and Arzel\`a-Ascoli theorem,  we obtain $g(t)=\lim_{i\to +\infty}g_i(t)$ on $M\times (0,S]$ with
\begin{enumerate}
\item[(a)] $(2\Lambda)^{-1}h\leq g(t)\leq 2\Lambda h$ on $M\times (0,S]$;
\item[(b)] for all $k\in \mathbb{N}$, there is $C(k,n,\Lambda_0)>0$ such that on $(0,S]$;
$$\sup_M |\tilde\nabla^k g(t)|\leq C(n,m,\Lambda_0) t^{-k/2}$$
\item[(c)] for all $x\in M$ and $t\in (0,S]$,
\begin{equation*}
\left(\fint_{B_h(x,1)} |\tilde\nabla g(t)|^n d\mathrm{vol}_h\right)^{1/n} \leq  (C_n\e^n+C_{n}L_1t)^{1/n}.
\end{equation*}
\end{enumerate}
The conclusion (d) follows from Lemma~\ref{lma:localEstimate-fromC^1conc} while the conclusion (e) follows from Proposition~\ref{prop:local-smoothness-Ck}.
\medskip

Now it remains to prove (f).  By a covering argument, it is sufficient to work on $B_h(x,1)$ for fixed $x\in M$. Proposition \ref{timezero-regularity-W1n} shows
\begin{equation}\label{Ln convergence}
\left(\fint_{B_h(x,1)} |g_{i}(t)-g_{i,0}|^n d\mathrm{vol}_h\right)^{1/n} \leq L_{n}t.
\end{equation}
For any $\ve>0$, there is a smooth metric $\bar g$ on $M$ such that
\begin{equation}\label{W1n convergence eqn 1}
\left(\fint_{B_h(x,1)} |\tilde\nabla(g_0-\bar g)|^n d\mathrm{vol}_h\right)^{1/n} \leq \ve.
\end{equation}
Recall that $g_{i,0}\to g_{0}$ in $W^{1,n}_{\loc}$. Then for sufficiently large $i$,
\begin{equation}
\left(\fint_{B_h(x,1)} |\tilde\nabla(g_{i,0}-\bar g)|^n d\mathrm{vol}_h\right)^{1/n} \leq 2\ve.
\end{equation}
By Proposition \ref{W1n stability} and a covering argument, there are $\ov{L}(n,\Lambda,\bar{g})$ and $C_{1}(n,\Lambda)$ such that
\begin{equation}
\left(\fint_{B_h(x,1)} |\tinabla(g_{i}(t)-\bar g)|^n d\mathrm{vol}_h\right)^{1/n} \leq C_{1}\ve+\ov{L}t.
\end{equation}
Combining this with \eqref{W1n convergence eqn 1},
\begin{equation}\label{W1n convergence}
\left(\fint_{B_h(x,1)} |\tinabla(g_{i}(t)-g_{0})|^n d\mathrm{vol}_h\right)^{1/n} \leq (C_{1}+1)\ve+\ov{L}t.
\end{equation}
In \eqref{Ln convergence} and \eqref{W1n convergence}, by letting $i\to\infty$, $t\to0$ and followed by $\ve\to0$, we obtain $g(t)\to g_{0}$ in $W^{1,n}_{\loc}$.
\end{proof}

Next, we will show that the solution of the Ricci-Deturck $h$-flow constructed in Theorem~\ref{thm-existence} is unique in the corresponding class. This will be crucial to study the persistence of scalar curvature under $W^{1,n}_{\loc}$ convergence.
\begin{thm}\label{thm-unique}
Suppose $h$ is a background metric satisfying $|\Rm(h)|\leq 1$. Then for any $\Lambda_0>1$, there is $\e_0(n,\Lambda_{0})>0$ such that the following holds.  If $g_0$ is a $L^\infty\cap W^{1,n}_{\loc}$ Riemannian metric  on $M$ such that
\begin{enumerate}
\item[(i)] $\Lambda_0^{-1}h\leq g_0\leq \Lambda_0 h$ on $M$;
\item[(ii)] for all $x\in M$,
$$\left(\fint_{B_h(x,1)} |\tilde \nabla g_0|^n \; d\mathrm{vol}_h \right)^{1/n}\leq \e_0$$
\end{enumerate}
then the solution to the Ricci-Deturck $h$-flow on $M\times (0,T]$ is unique within the class of solution $g(t)$ satisfying
\begin{enumerate}
\item[(a)] $(\Upsilon_n\Lambda_{0})^{-1} h\leq g(t)\leq \Upsilon_n\Lambda_{0} h$ on $M\times (0,T]$;
\item[(b)]  For all $x\in M$ and $t\in (0,T]$,
\[
\left(\fint_{B_h(x,1)} |\tilde \nabla g(t)|^n \; d\mathrm{vol}_h \right)^{1/n}\leq C(n,\Lambda_0)  \e_0;
\]
\item[(c)]  $g(t)\to g_0$ in $L^2_{\loc}$ as $t\to 0$,
\end{enumerate}
where $\Upsilon_n$ is the constant obtained in Lemma \ref{borderline-trace}.
\end{thm}
\begin{proof}
Suppose $g(t)$ and $\hat{g}(t)$ are two solutions of $h$-flow satisfying (a), (b) and (c). We will use $C_i$ to denote constants depending only on $n$ and $\Lambda_{0}$. By the Ricci-Deturck $h$-flow equation \eqref{eqn: h-flow}, $\a(t)=g(t)-\hat g(t)$ satisfies
\begin{equation}
\begin{split}
\partial_t \a_{ij}&=g^{pq}\tilde\nabla_p \tilde\nabla_q g_{ij}-\hat g^{pq}\tilde\nabla_p \tilde\nabla_q \hat g_{ij}+ \widetilde{\mathrm{Rm}}*g^{-1}*\hat g^{-1}* g*\a \\
&\quad +  \widetilde{\mathrm{Rm}}*\hat g^{-1}*\a + g^{-1}*\hat g^{-1}*g^{-1}* \tilde\nabla g*\tilde\nabla g* \a\\
&\quad + \hat g^{-1} *\hat g^{-1} * g^{-1}*\tilde\nabla g*\tilde\nabla g*\a +  \hat g^{-1} *\hat g^{-1} *\tilde\nabla g*\tilde\nabla \a\\
&\quad +  \hat g^{-1} *\hat g^{-1} *\tilde\nabla \hat g*\tilde\nabla \a.
\end{split}
\end{equation}
Let $\eta$ be a smooth cutoff function on $M$ such that $\eta\equiv 1$ on $B_h(x,1)$, $\eta\equiv0$ on $M\setminus B_h(x,2)$ and $|\tilde\nabla\eta|\leq 10^4$. Using assumption (i) and the Cauchy-Schwarz inequality, for any $\tau\in(0,1)$,
\begin{equation}
\begin{split}
& \frac{\partial}{\partial t}\int_M \eta^{4} |\a|^2 \; d\mathrm{vol}_h\\
\leq {} & \int_M 2\eta^{4}h^{ik}h^{jl} \a_{kl} \left[g^{pq}\tilde\nabla_p \tilde\nabla_q \a_{ij}+\left(g^{pq}-\hat g^{pq}\right)\tilde\nabla_p \tilde\nabla_q \hat g_{ij}\right]\\
&+C_1\tau^{-1} \int_M \eta^{4} |\a|^2(1+|\tilde \nabla g|^2+|\tilde \nabla \hat{g}|^2) \; d\mathrm{vol}_h+ \tau  \int_M \eta^{4} |\tilde \nabla \a|^2 \; d\mathrm{vol}_h\\
= {} & -2\int_M h^{ik}h^{jl}\tilde\nabla_p\left( \eta^{4}\a_{kl} g^{pq} \right) \tilde\nabla_q \a_{ij}\; d\mathrm{vol}_h
+\int_M \eta^{4} g^{-1}*\hat g^{-1}*\a * \a *\tilde\nabla^2 \hat g  \, d\mathrm{vol}_h\\
&+C_1\tau^{-1} \int_M \eta^{4} |\a|^2(1+|\tilde \nabla g|^2+|\tilde \nabla \hat{g}|^2) \; d\mathrm{vol}_h+ \tau \int_M \eta^{4} |\tilde \nabla \a|^2 \; d\mathrm{vol}_h\\
\leq {} & -C_2^{-1} \int_M \eta^{4}|\tilde\nabla \a|^2 \;d\mathrm{vol}_h+ \int_M \eta^{4} g^{-1}*\hat g^{-1}*\a * \a *\tilde\nabla^2 \hat g  \, d\mathrm{vol}_h\\
&+ C_2\int_M \eta^{2} |\a|^2  \;d\mathrm{vol}_h+C_2\tau^{-1}\int_M \eta^{4} |\a|^2(1+|\tilde \nabla g|^2+|\tilde \nabla \hat{g}|^2)\, d\mathrm{vol}_h
+\tau \int_M \eta^{4} |\tilde \nabla \a|^2 \; d\mathrm{vol}_h.
\end{split}
\end{equation}
Choosing $\tau=(2C_{2})^{-1}$, we obtain
\begin{equation}\label{thm-unique eqn 1}
\begin{split}
& \frac{\partial}{\partial t}\int_M \eta^{4} |\a|^2 \; d\mathrm{vol}_h \\
\leq {} & -C_3^{-1} \int_M \eta^{4}|\tilde\nabla \a|^2 \;d\mathrm{vol}_h+ \int_M \eta^{4} g^{-1}*\hat g^{-1}*\a * \a *\tilde\nabla^2 \hat g  \, d\mathrm{vol}_h\\
&+ C_3\int_M \eta^{2} |\a|^2  \;d\mathrm{vol}_h+C_3\int_M \eta^{4} |\a|^2(1+|\tilde \nabla g|^2+|\tilde \nabla \hat{g}|^2)\, d\mathrm{vol}_h.
\end{split}
\end{equation}
For the second term on the right hand side of \eqref{thm-unique eqn 1}, integrating by parts and using the Cauchy-Schwarz inequality, for any $\tau\in(0,1)$,
\begin{equation}\label{thm-unique eqn 3}
\begin{split}
&\quad \int_M \eta^{4} g^{-1}*\hat g^{-1}*\a * \a *\tilde\nabla^2 \hat g  \, d\mathrm{vol}_h\\
&= \int_M \tilde \nabla g * \tilde\nabla (\eta^{4} g^{-1} *\hat g^{-1} * \a *\a) \,d\mathrm{vol}_h\\
&\leq C_4 \int_{M} |\tilde\nabla g| \left(\eta^{3}|\a|^2+\eta^{4}|\tilde\nabla g||\a|^2+\eta^{4}|\tilde\nabla \hat g||\a|^2+ \eta^{4} |\a| |\tilde\nabla \a| \right)d\vol_{h}\\
&\leq  C_5 \int_{M}\left(\tau^{-1}\eta^{4}|\a|^2 (|\tilde \nabla g|^{2}+|\tilde\nabla \hat g|^{2})+ \eta^{2}|\a|^2 \right)d\vol_{h}
+ \tau\int_{M} \eta^{4}|\tilde \nabla \a|^2\;d\vol_{h}
\end{split}
\end{equation}
Choosing $\tau=(2C_{3})^{-1}$ and substituting \eqref{thm-unique eqn 3} into \eqref{thm-unique eqn 1},
\begin{equation}\label{thm-unique eqn 6}
\begin{split}
\frac{\partial}{\partial t}\int_M \eta^{4} |\a|^2 \; d\mathrm{vol}_h
\leq {} & -C_6^{-1} \int_M \eta^{4}|\tilde\nabla \a|^2 \;d\mathrm{vol}_h
+C_6\int_M |\a|^2 \eta^{2} \;d\mathrm{vol}_h \\
& +C_6\int_M \eta^{4}|\a|^2(1+|\tilde \nabla g|^2+|\tilde\nabla\hat g|^2) \; d\mathrm{vol}_h.
\end{split}
\end{equation}
We squeeze more negativity from the first term by using the Sobolev inequality \eqref{Sobo-ineq} and Kato's inequality:
\begin{equation}
\begin{split}
|B_h(x,2)|^{2/n} \left(\int_M (\eta^{4}|\a|^2)^\frac{n}{n-2} \right)^\frac{n-2}{n}
\leq {} & C_{7}\int_{M}|\tinabla(\eta^{2}|\alpha|)|^{2}d\vol_{h} \\
\leq {} & C_{7}\int_{M}(\eta^{2}|\alpha|^{2}+\eta^{4}|\tinabla\alpha|^{2})d\vol_{h}
\end{split}
\end{equation}
and then
\begin{equation}\label{thm-unique eqn 2}
\begin{split}
& -C_6^{-1} \int_M \eta^{4}|\tilde\nabla \a|^2 \;d\mathrm{vol}_h \\
\leq {} & -C_{8}^{-1}|B_h(x,2)|^{2/n} \left(\int_M (\eta^{4}|\a|^2)^\frac{n}{n-2} \right)^\frac{n-2}{n}
+ C_8\int_M \eta^{2} |\a|^2  \;d\mathrm{vol}_h.
\end{split}
\end{equation}
Substituting \eqref{thm-unique eqn 2} into \eqref{thm-unique eqn 6},
\begin{equation}\label{thm-unique eqn 4}
\begin{split}
& \frac{\partial}{\partial t}\int_M \eta^{4} |\a|^2 \; d\mathrm{vol}_h\\
\leq {} & -C_9^{-1} |B_h(x,2)|^{2/n}\left(\int_M (\eta^{4}|\a|^2)^\frac{n}{n-2} \;d\mathrm{vol}_h\right)^\frac{n-2}{n}\\
& + C_9\int_M |\a|^2 \eta^{2} \;d\mathrm{vol}_h+C_9\int_M \eta^{4}|\a|^2(1+|\tilde \nabla g|^2+|\tilde\nabla\hat g|^2) \; d\mathrm{vol}_h.
\end{split}
\end{equation}
By assumption (ii) and a covering argument,
\begin{equation}
\begin{split}
&  \int_M \eta^{4} |\a|^2 |\tilde\nabla g|^2 \;d\mathrm{vol}_h\\
\leq {} & \left(\int_M \left(\eta^{4}|\a|^2\right)^\frac{n}{n-2} \, d\mathrm{vol}_h \right)^\frac{n-2}{n} \left(\fint_{B_h(x,2)} |\tilde\nabla g|^{n} \, d\mathrm{vol}_h \right)^{2/n} |B_h(x,2)|^{2/n}\\
\leq {} & C_n\e_0 \left(\int_M \left(\eta^{4} |\a|^2\right)^\frac{n}{n-2} \, d\mathrm{vol}_h \right)^\frac{n-2}{n} |B_h(x,2)|^{2/n}.
\end{split}
\end{equation}
Similarly,
\begin{equation}
\int_M \eta^{4} |\a|^2 |\hat\nabla g|^2 \;d\mathrm{vol}_h
\leq C_n\e_0 \left(\int_M \left(\eta^{4} |\a|^2\right)^\frac{n}{n-2} \, d\mathrm{vol}_h \right)^\frac{n-2}{n}\cdot |B_h(x,2)|^{2/n}.
\end{equation}
Therefore, if $\e_0$ is sufficiently small depending only on $n$ and $\Lambda_{0}$, then \eqref{thm-unique eqn 4} can be reduced to
\begin{equation}\label{evo-difference}
\begin{split}
 \frac{\partial}{\partial t}\int_M \eta^{4} |\a|^2 \; d\mathrm{vol}_h
&\leq C_{10}\int_M \eta^{2}|\a|^2  \;d\mathrm{vol}_h+C_{10}\int_M \eta^{4} |\a|^2 \; d\mathrm{vol}_h.
\end{split}
\end{equation}

For any $\e>0$, let $S_{\e}$ be the maximal time such that for all $(z,t)\in M\times (0,S_\e]$,
\begin{equation}\label{ap-diff-bound}
\int_{B_h(z,1)}|\a|^2\, d\mathrm{vol}_h<\e.
\end{equation}
Proposition~\ref{timezero-regularity-W1n} shows $S_{\e}>0$. Then there is $x_{0}\in M$ such that
\begin{equation}\label{thm-unique eqn 5}
\left(\int_{B_h(x_{0},1)}|\a|^2\, d\mathrm{vol}_h\right)\bigg|_{t=S_{\e}} \geq \frac{\e}{2}.
\end{equation}
We claim that $S_\e$ is bounded from below uniformly (independent of $\e$). Combining \eqref{evo-difference} with \eqref{ap-diff-bound}, and using a covering argument, for all $t\in (0,S_\e]$,
\begin{equation}
\frac{\partial}{\partial t}\left(e^{-C_{10}t}\int_M \eta^{4}|\a|^2 \; d\mathrm{vol}_h\right) \leq C_{11}\e e^{-C_{10}t} \leq C_{11}\e.
\end{equation}
Integrating this on $[s,S_\e]$, letting $s\to 0$, and using the fact that $g(t),\hat g(t)\to g_0$ in $L^2_{\loc}$ as $t\to 0$, we conclude that
\begin{equation}
\left(\int_{B_h(x,1)} |\a|^2 \, d\mathrm{vol}_h\right)\bigg|_{t=S_\e}\leq C_{11}\e e^{C_{10}S_\e} S_\e.
\end{equation}
Combining this with \eqref{thm-unique eqn 5}, we obtain $S_\e\geq S_0(n,\Lambda)>0$ which is independent of $\e$.  This gives the uniqueness on $M\times (0,S_0]$ as $\e$ is arbitrary.  The uniqueness beyond $S_0$ is standard thanks to the smoothness of $g(t)$ for $t>0$.
\end{proof}

\medskip

We now show that if $(M,h)$ is non-collapsed at infinity, then the small gradient concentration assumption in Theorem~\ref{thm-existence} can always be achieved by scaling if in addition $g_0\in W^{1,n}$ globally.
\begin{prop}\label{scaling-assumption}
Suppose $h$ is a smooth metric satisfying assumption $(\star)$. If in addition $\mathrm{inj}(M,h)>0$ and $g_0\in W^{1,n}$,  then for any $\e>0$, there is $r_\e\in(0,1)$ such that for all $x\in M$, we have
\begin{equation}
\left( r_{\e}^n \fint_{B_h(x,r_{\e})} |\tilde\nabla g_0|^n\; d\mathrm{vol}_h\right)^{1/n}<\e.
\end{equation}
\end{prop}
\begin{proof}
By assumption on injectivity radius of $h$, there is $v_0>0$ such that for all $x\in M$,
\begin{equation}\label{scaling-assumption eqn 1}
\mathrm{Vol}_h\left( B_h(x,r)\right)\geq v_0r^n
\end{equation}
for all $r\in (0,1]$. Let $\e>0$ be fixed and $x_0\in M$. By assumption $g_{0}\in W^{1,n}$, there is $R>1$ such that
\begin{equation}\label{scaling-assumption eqn 2}
\left(\int_{M\setminus B_h(x_0,R)} |\tilde\nabla g_0|^n \; d\mathrm{vol}_h\right)^{1/n} <\e .
\end{equation}
We claim that there is $r_0(\e)\in(0,1)$ such that for all $x\in B_h(x_0,R+1)$,
\begin{equation}\label{scaling-assumption eqn 3}
\left(\int_{B_h(x,r_0)} |\tilde\nabla g_0|^n \; d\mathrm{vol}_h  \right)^{1/n}<\e.
\end{equation}
Suppose the claim is false, there is a sequence of $x_i\in B_h(x_0,R+1)$ and $r_i\to 0$ such that
\begin{equation}
\left(\int_{B_h(x_i,r_i)} |\tilde\nabla g_0|^n \; d\mathrm{vol}_h  \right)^{1/n} \geq \e.
\end{equation}
We may assume $x_i\to x_\infty\in B_h(x_0,R+2)$. Since $r_i\to 0$, then $\lim_{i\to +\infty}\chi_{B_h(x_i,r_i)}= \chi_{\{x_\infty\}}$ and Lebesgue's dominated convergence theorem implies that
\begin{equation}
\lim_{i\to +\infty}\left(\int_{B_h(x_i,r_i)} |\tilde\nabla g_0|^n \; d\mathrm{vol}_h  \right)^{1/n} =0,
\end{equation}
which is impossible.

On the other hand, for all $x\in M\setminus B_h(x_0,R+1)$,  we have $B_h(x,1)\Subset M\setminus B_h(x_0,R)$. By \eqref{scaling-assumption eqn 2},
\begin{equation}\label{scaling-assumption eqn 4}
\left(\int_{B_h(x,1)}  |\tilde\nabla g_0|^n \; d\mathrm{vol}_h  \right)^{1/n}
\leq \left(\int_{M\setminus B_h(x_0,R)}  |\tilde\nabla g_0|^n \; d\mathrm{vol}_h  \right)^{1/n} < \e.
\end{equation}
Then the conclusion follows from \eqref{scaling-assumption eqn 3} and \eqref{scaling-assumption eqn 4} by relabelling the constant, using \eqref{scaling-assumption eqn 1} and volume comparison theorem.
\end{proof}

Proposition~\ref{scaling-assumption} in particular says that the rescaled metric $\hat g_0=r_{\e}^{-2}g_0$ and $\hat h=r_{\e}^{-2}h$ will remain bi-Lipschitz while the small gradient concentration can be achieved everywhere so that  Theorem~\ref{thm-existence} applied to obtain a global Ricci-Deturck $\hat h$-flow $\hat g(t)$ starting from $\hat g_0$ (in $W^{1,n}_{\loc}$ sense) with scaling invariant estimates.

\begin{thm}\label{shorttime-compact}
Let $(M^n,h)$ be a complete Riemannian manifold with bounded curvature and $\mathrm{inj}(M,h)>0$. Suppose $g_0$ is a $L^\infty\cap W^{1,n}$ Riemannian metric on $M$ such that
\begin{equation}
\Lambda_0^{-1}h\leq g_0\leq \Lambda_0 h
\end{equation}
for some $\Lambda_0>1$. Then there is a smooth solution $g(t)$ to the Ricci-Deturck $h$-flow on $M\times (0,S]$ for some $S>0$ such that
\begin{enumerate}
\item[(a)] $(\Upsilon_n\Lambda_0)^{-1} h\leq g(t)\leq \Upsilon_n\Lambda_0 h$ on $M\times (0,S]$;
\item[(b)] For any $k\in \mathbb{N}$,
\begin{equation*}
\limsup_{t\to 0^+}\left(\sup_M t^{k/2}|\tilde\nabla^k g(t)|\right)= 0;
\end{equation*}
\item[(c)] If $g_0\in C^\infty_{\loc}(\Omega)$ for some $\Omega\Subset M$, then $g(t)\to g_0$ in $C^\infty_{\loc}(\Omega)$ as $t\to 0$;
\item[(d)] $g(t)\to g_0$ in $W^{1,n}_{\loc}$ as $t\to 0$,
\end{enumerate}
where $\Upsilon_n$ is the constant obtained in Lemma \ref{borderline-trace}.
\end{thm}
\begin{proof}
By scaling, the existence and conclusion (a), (c), (d) follow from Proposition~\ref{scaling-assumption} and Theorem~\ref{thm-existence}. It remains to show (b).

For any $\e>0$, let $r_\e$ be the constant obtained in Proposition~\ref{scaling-assumption}. Consider the rescaled metric $\hat h=r_\e^{-2}h$ and $\hat g_0=r_\e^{-2}g_0$. Since $r_\e<1$, the new background metric still satisfies $(\star)$ and $\Lambda_0^{-1}\hat h\leq \hat g_0\leq \Lambda_0 \hat h$. In addition, we have
\begin{equation}
\sup_{x\in M} \left(\fint_{B_{\hat h}(x,1)}|\nabla^{\hat h}\hat g_0|^n \; d\mathrm{vol}_{\hat h} \right)^{1/n}<\e.
\end{equation}
Applying Theorem~\ref{thm-existence} on $\hat g_0$ and $\hat h$ and rescale them back to $g_0$ and $h$, we obtain a short-time solution $g(t)$ on $M\times (0,T\cdot r_\e^2]$ with
\begin{equation}\label{estimate-compact}
\left\{
\begin{array}{ll}
(\Upsilon_n\Lambda_0)^{-1}  h\leq  g(t)\leq (\Upsilon_n\Lambda_0)  h;\\[1mm]
\limsup_{t\to 0^+}\left(\sup_M t^{1/2}|\tilde \nabla g(t)|\right)\leq C(n,\Lambda_0)\e^{1/2};\\[2mm]
\sup_M t^{k/2}|\tilde \nabla^k g(t)|\leq C(n,k,\Lambda_0)
\end{array}
\right.
\end{equation}
and converges $g_0$ as $t\to 0$ in $W_{\loc}^{1,n}$.  Here we used the fact that the solution $g(t)$ is unique by Theorem~\ref{thm-unique} so that all solution constructed from rescaling coincides. Since $\e$ is arbitrarily,  we have
\begin{equation}\label{1st-ORDER}
\limsup_{t\to 0^+}\left(\sup_M t^{1/2}|\tilde \nabla g(t)|\right)=0.
\end{equation}

It remains to show that the same holds for $|\tilde\nabla^kg(t)|$, $k>1$.  We start with $k=2$. We claim that
\begin{equation}
\limsup_{t\to 0^+}\left(\sup_M t|\tilde \nabla^2 g(t)|\right)=0.
\end{equation}
Suppose the assertion is false, there is a sequence $x_i\in M$ and $t_i\to 0$ such that
\begin{equation}\label{shorttime-compact eqn 1}
t_i|\tilde\nabla^2 g|(x_i,t_i)> \e_1
\end{equation}
for some $\e_1>0$. Consider $g_i=t_i^{-1}g(t_i)$ and $h_i=t_i^{-1}h$.  By the assumptions of $h$ and $t_i\to 0$,  $(M,h_i,x_i)$ converges to the standard Euclidean space $(\mathbb{R}^n,g_{\mathrm{flat}},0)$ in the smooth pointed Cheeger-Gromov sense, while \eqref{estimate-compact} shows that $g_i$ converges to some smooth metric $g_\infty$ on $\mathbb{R}^n$ which is uniformly equivalent to $g_{\mathrm{flat}}$. By \eqref{shorttime-compact eqn 1}, we obtain $|D^2 g_\infty(0)|>0$.  However, \eqref{1st-ORDER} implies that $Dg_\infty\equiv 0$ on $\mathbb{R}^n$, which is impossible.  The higher order asymptotic can be proved similarly by using induction argument. This completes the proof.
\end{proof}

\begin{rem}
One might compare the result with the work of Miles \cite{Simon2002} where $C^0$ assumption is now replaced by $W^{1,n}$. Both assumptions are also imposed to rule out the possibility of metric cone as the singular model (the possible limiting case). This is reflected in curvature estimates (b).
\end{rem}

\section{Stability on Euclidean space}\label{Sec:stability}

In this section,  we consider the case when $h$ is flat so that the assumption $(\star)$ is scaling invariant for all scale. We will study the asymptotic behaviour of the constructed Ricci-Deturck $h$-flow as $t\to +\infty$. The following long-time existence follows directly from scaling argument.
\begin{thm}\label{Thm:LT}
For any $\Lambda_0>1$, there is $\e_0(n,\Lambda_{0})>0$ such that the following holds.  Suppose $(M,h)$ is a complete non-compact flat manifold. If $g_0\in L^\infty\cap W^{1,n}_{\loc}$ is a Riemannian metric (not necessarily smooth) such that
\begin{enumerate}
\item[(i)] $\displaystyle \Lambda_0^{-1}h \leq g_0\leq \Lambda_0 h$ on $M$;
\item[(ii)]
$\displaystyle \sup_{x\in M,r>1} \left(r^n\fint_{B_h(x,r)} |\tilde \nabla g_0|^n d\mathrm{vol}_h\right)^{1/n}\leq \e_0$ ;
\end{enumerate}
then there is a long-time solution $g(t)$ to the Ricci-Deturck $h$-flow on $M\times (0,+\infty)$ so that $g(t)\to g_0$ in $W^{1,n}_{\loc}$  as $t\to 0$ and for all $t\in (0,+\infty)$,
\begin{enumerate}
\item[(a)] $(C_n\Lambda_0)^{-1}h \leq g(t)\leq C_n\Lambda_0 h$ for some $C_n>1$;
\item[(b)] for all $k\in \mathbb{N}$, there is $C(n,k,\Lambda_0)>0$ such that for all $t\in (0,+\infty)$,
$$\sup_M |\tilde\nabla^kg(t)|^2\leq \frac{C(n,k,\Lambda_0)}{t^k}.$$
\item[(c)] for all $t\in (0,+\infty)$,
$$\sup_{x\in M,r>1}\left(r^n \fint_{B_h(x,r)} |\tilde\nabla g|^n\, d\mathrm{vol}_h \right)^{1/n} \leq C_n\e _0+C_0(n,\Lambda_0)(r^{-2}t)^{1/n}$$
\end{enumerate}
In particular, $g(t)$ sub-converges to some flat metric on $M$ as $t\to +\infty$.
\end{thm}
\begin{proof}
Theorem \ref{thm-existence} guarantees the short-time existence of the solution $g(t)$ to the Ricci-Deturck $h$-flow. We will show $g(t)$ exists on $(0,+\infty)$ and satisfies (a), (b) and (c).

For any $i\in\mathbb{N}$, consider the rescaled metric $g_{i,0}=i^{-2}g_0$ and $h_i=i^{-2}h$.  By Theorem~\ref{thm-existence}, there is a short-time solution $g_i(t)$ to the Ricci-Deturck $h_i$-flow on $M\times (0,T]$ with
 \begin{enumerate}
\item[(a')] $(C_n\Lambda_0)^{-1}h_i \leq g_i(t)\leq C_n\Lambda_0 h_i$ for some $C_n>1$;
\item[(b')] For all $k\in \mathbb{N}$ and $t\in (0,T]$,
$$\sup_M |\nabla^{h_i,k}g_i(t)|_{h_i}^2\leq \frac{C(n,k,\Lambda_0)}{t^k};$$
\item[(c')] $g_i(t)\to g_{i,0}$ in $W^{1,n}_{\loc}$  as $t\to 0$.
\end{enumerate}
Rescale the flow $g_i(t)$ back to $\hat{g}_{i}(t)=i^2g_i(i^{-2}t)$, which is defined on $(0,i^2T]$. Since the solution is unique by Theorem~\ref{thm-unique}, $g(t)=\hat g_i(t)$ for any $i\in\mathbb{N}$ and hence $g(t)$ exists for all $t>0$. Then (a) and (b) follow from (a') and (b') since the derivatives are scaling invariant. To establish (c), we apply Theorem \ref{thm-existence} the rescaled flow $g_i(t)$ and obtain
\begin{equation}
\begin{split}
\left(i^n \fint_{B_{h}(x,i)} |\tilde\nabla g(i^2t)|^n \, d\mathrm{vol}_{h}\right)^{1/n}
= {} & \left( \fint_{B_{h_i}(x,1)} |\nabla^{h_i}g_i(t)|^n \, d\mathrm{vol}_{h_i}\right)^{1/n}\\[2mm]
\leq {} & C_n\e_0 +C_0t^{1/n}
\end{split}
\end{equation}
for each $x\in M$, $t\in (0,T]$ and $i\in\mathbb{N}$. Relabelling the parameter and using the volume comparison theorem, we obtain (c). The convergence follows from the Arzel\`a-Ascoli theorem and the derivatives estimates.
\end{proof}

\begin{rem}
If $(M,h)=(\mathbb{R}^n,g_{\mathrm{flat}})$ is the standard Euclidean space, then the assumption (ii) is equivalent to the smallness of global gradient concentration: $ \int_{M} |\tilde\nabla g|^n \, d\mathrm{vol}_h<\omega_{n}\e_{0}^{n}$, where $\omega_{n}$ denotes the volume of unit ball in $\mathbb{R}^{n}$.
\end{rem}

Inspired by the works of \cite{SSS2008}, we now consider the behaviour of $g(t)$ as $t\to +\infty$ if in addition $g_0$ is sufficiently close to $h$ at infinity.
\begin{thm}
Under the assumption of Theorem~\ref{Thm:LT}, if in addition $(M,h)=(\mathbb{R}^n,g_{\mathrm{flat}})$ is the standard Euclidean space and $|g_0-h|\in L^2(\mathbb{R}^n,g_{\mathrm{flat}})$, then $g(t)\to g_{\mathrm{flat}}$ in $C^\infty_{\loc}$ as $t\to +\infty$.
\end{thm}
\begin{proof}
Consider the difference $\b(t)=g(t)-h$ which satisfies
\begin{equation}
\begin{split}
\frac{\partial }{\partial t}\b_{ij}&= g^{pq}\tilde\nabla_p \tilde\nabla_q \b_{ij}+ g^{-1}*g^{-1}* \tilde \nabla \b* \tilde \nabla \b.
\end{split}
\end{equation}
since $\tilde\nabla h=0$.   Let $\phi$ be a cut-off function on $[0,+\infty)$ such that $\rho\equiv1$ on $[0,1]$, $\rho\equiv0$ on $[2,+\infty]$ and $|\phi'|\leq10$, and $\rho$ be a function on $M=\mathbb{R}^n$ which is equivalent to the distance function from a fixed point $p\in M=\mathbb{R}^n$ and satisfies $\|\rho\|_{C^2(M,h)}\leq 1$. Define $\eta_R=\phi^{n}\left(\frac{\rho}{R} \right)$ and consider the energy
$$E_R(t)=\int_M |\b|^2 \eta^2_R\, d\mathrm{vol}_h.$$
Then our assumption is equivalent to say that there is $E_0>0$ so that $E_R(0)\leq E_0$ for all $R>1$.  We will use $C_i$ to denote constants depending only on $n$ and $\Lambda_{0}$.

Differentiating $E_R$ with respect to $t$ and integrating by parts yields
\begin{equation}\label{E R t estimate}
\begin{split}
\frac{d}{dt}E_R(t)&= \int_M \left(2\langle \b, g^{pq}\tilde\nabla_p \tilde\nabla_q \b \rangle + g^{-1}*g^{-1}* \tilde \nabla \b* \tilde \nabla \b*\b \right) \eta^2_R \, d\mathrm{vol}_h\\
&\leq \int_M \left( -C_{1}^{-1}\eta_R^2|\tilde \nabla \b|^2 +C_1|\b| |\tilde\nabla \b|^2 \eta^2_R+C_1|\tilde\nabla \eta_R|^2 |\b|^2\right)d\mathrm{vol}_h\\
&\leq \int_M \left(-C_{2}^{-1}|\tilde \nabla (\eta_R\b)|^2 +C_2|\b|^2 |\tilde\nabla \b|^2 \eta^2_R+C_2|\tilde\nabla \eta_R|^2 |\b|^2\right)d\mathrm{vol}_h\\[2mm]
&=-\mathbf{I}+\mathbf{II}+\mathbf{III}.
\end{split}
\end{equation}
We squeeze more negativity from $\mathbf{I}$ by using the Sobolev inequality \eqref{Sobo-ineq} and Kato's inequality:
\begin{equation}
\mathbf{I}\geq C_{3}^{-1} \left(\int_M |\b \eta_R|^\frac{2n}{n-2}\, d\mathrm{vol}_h \right)^\frac{n-2}{n},
\end{equation}
while $\mathbf{II}$ can be estimated using (c) in  Theorem~\ref{Thm:LT} (with $r=+\infty$):
\begin{equation}
\begin{split}
\int_M |\b|^2 |\tilde \nabla \b|^2 \eta_R^2 \,d\mathrm{vol}_h
&\leq  \left(\int_M |\b \eta_R|^\frac{2n}{n-2}\, d\mathrm{vol}_h \right)^\frac{n-2}{n} \left(\int_M |\tilde\nabla g|^{n}\,d\mathrm{vol}_h \right)^{1/n}\\
&\leq C_n\e_0\left(\int_M |\b \eta_R|^\frac{2n}{n-2}\, d\mathrm{vol}_h \right)^\frac{n-2}{n}.
\end{split}
\end{equation}
Therefore, if $\e_0$ in the assumption is sufficiently small, $-\mathbf{I}+\mathbf{II}\leq 0$ for any $R>1$.

On the other hand, since $g(t)$ is uniformly equivalent to $h$, then $|\beta|\leq C$. Combining this with $\eta_{R}=\phi^{n}\left(\frac{\rho}{R}\right)$,
\begin{equation}\label{III estimate}
\begin{split}
\mathbf{III}&\leq \frac{C_{4}n^2}{R^2}\int_M  |\b|^2 \eta_R^{2-\frac2n} \ d\mathrm{vol}_h\\
&\leq \frac{C_{5}}{R^2}\int_M  |\b|^{2-\frac2n} \eta_R^{2-\frac2n} \ d\mathrm{vol}_h\\
&\leq \frac{C_{5}}{R^{2}} E_R^{1-\frac1{n}}|B_{h}(x,2R)|^{\frac{1}{n}} \\
&\leq \frac{C_{6}}{R} E_R^{1-\frac1{n}}.
\end{split}
\end{equation}
Here we used Young's inequality and the volume growth of Euclidean space. Substituting $-\mathbf{I}+\mathbf{II}\leq 0$ and \eqref{III estimate} into \eqref{E R t estimate},
\begin{equation}
\frac{d}{dt}E_R(t) \leq \frac{C_{6}}{R} E_R^{1-\frac1{n}}.
\end{equation}
For any $\sigma>0$, the above shows
\begin{equation}
\frac{d}{dt}(E_R(t)+\sigma)^{\frac{1}{n}} = \frac{C_{6}}{nR}(E_R(t)+\sigma)^{\frac{1}{n}-1} E_R^{1-\frac1{n}} \leq \frac{C_{6}}{nR}.
\end{equation}
Integrating on $[s,t]$ and letting $\sigma\to0$,
\begin{equation}
E^{\frac{1}{n}}_R(t) \leq E^{\frac{1}{n}}_R(s)+\frac{C_{6}(t-s)}{nR}.
\end{equation}
Letting $s\to0$ and followed by $R\to+\infty$, we see that for all $t>0$,
\begin{equation}\label{L2-longtime}
\int_M |\b(t)|^2\,d\mathrm{vol}_h\leq  \int_M |\b(0)|^2\,d\mathrm{vol}_h.
\end{equation}
By the Arzel\`a-Ascoli theorem and the derivatives estimates, passing to a subsequence, $g(t)$ converges to some flat metric $\hat h$ in $C^\infty_{\loc}$. By Lebesgue's dominated convergence theorem, \eqref{L2-longtime} implies
\begin{equation}
\int_M |h-\hat{h}|^2\,d\mathrm{vol}_h < +\infty.
\end{equation}
Since $(M,h)$ is the standard Euclidean space, then $\hat h\equiv h$ on $M$. Hence, the above argument shows that $g(t)\to h$ as $t\to +\infty$ in $C^\infty_{\loc}$ without passing to subsequence. This completes the proof.
\end{proof}

\section{Applications on almost rigidity problems}\label{Sec:almostRid}

In this section, we will use the Ricci-Deturck $h$-flow to study scalar curvature compactness problems under assumptions of uniformly bi-Lipschitz and $W^{1,n}_{\loc}$ smallness. This is motivated by the torus stability problems, see \cite{Gromov2014,Sormani2021}. We begin with recalling the notion of Yamabe invariant $\sigma(M)$ which is given by
\begin{equation}
\sigma(M)=\sup\left\{ \mathcal{Y}(M,[g]): [g] \textit{ is a conformal class of metrics on } M\right\},
\end{equation}
where
\begin{equation}
\mathcal{Y}(M,[g_0])=\inf\left\{ \int_M R(g)\;d\mu_g: g\in[g_0],\; \mathrm{Vol}(M,g)=1\right\}.
\end{equation}

It is well known that if a smooth metric on a
compact manifold attains the Yamabe invariant $\sigma(M)$ and if $\sigma(M)\leq 0$, then the metric must be Einstein \cite{Schoen1989}.  The basic model of manifolds with $\sigma(M)=0$ is the standard torus $\mathbb{T}^n$ which was shown by Schoen-Yau \cite{SchoenYau1979,SchoenYau1979-2} when $n\leq 7$ and Gromov-Lawson \cite{GromovLawson1980} in general dimension.

We use the geometric flow smoothing to establish an integral stability on manifolds with non-positive Yamabe invariant, which in particular is applicable to the torus.  This is in spirit similar to that in \cite{LeeNaberNeumayer2020}.

\begin{thm}\label{thm:stability-sigma}
Let $M$ be a compact manifold with $\sigma(M)\leq 0$ with a background metric $h$ satisfying $(\star)$.  For any $\Lambda_0>1$, there is $\e_0(n,\Lambda_0)>0$ such that the following  holds. Suppose $g_{i,0}$ is a sequence of Riemannian metric on $M$ such that
\begin{enumerate}
\item[(i)] $\Lambda_0^{-1}h\leq g_{i,0}\leq \Lambda_0 h$ on $M$;
\item[(ii)] $\displaystyle \left( \fint_{B_h(x,1)} |\tilde \nabla g_{i,0}|^n \, d\mathrm{vol}_h \right)^{1/n} \leq \e_0$ for all $x\in M$;
\item[(iii)] $\mathcal{R}(g_{i,0})\geq -i^{-1}$ on $M$.
\end{enumerate}
Then there is a Ricci-flat metric $\hat h$ on $M$ so that after passing to subsequence, $g_{i,0}$ converge to $\hat h$ in $ L^p$ for all $p>0$ as $i\to +\infty$, modulo diffeomorphism.
\end{thm}

\begin{proof}
By Theorem~\ref{thm-existence}, there exists a sequence of Ricci-Deturck $h$-flow $g_i(t)$ on $M\times [0,T]$ starting from $g_{i,0}$ such that
\begin{enumerate}
\item[(a)] $\Lambda_{0}^{-1}h\leq g_i(t)\leq \Lambda_{0} h$;
\item[(b)] For any $k\in \mathbb{N}$, there is $C(k,n,\Lambda_0)>0$ such that for all $t\in (0,T]$,
\begin{equation*}
\sup_M |\tilde\nabla^k g_i(t)|\leq C(k,n,\Lambda_0) t^{-k/2};
\end{equation*}
\item[(c)] $\displaystyle \left(\fint_{B_h(x,1)} |\tilde \nabla g_i(t)|^n \,d\mathrm{vol}_h \right)^{1/n}\leq C_0\e_0$ for all $(x,t)\in M\times [0,T]$.
\end{enumerate}
Moreover, it is well-known that the Ricci-Deturck flow preserves the scalar curvature lower bound and hence for all $t\in [0,T]$ and $i\in \mathbb{N}$,
\begin{equation}
\mathcal{R}(g_i(t))\geq -i^{-1}.
\end{equation}
Therefore, using diagonal subsequence argument and the Arzel\`a-Ascoli theorem, after passing to subsequence, $g_i(t)\to g(t)$ on $M\times (0,T]$ in $C^\infty_{\loc}$ with
\begin{enumerate}
\item[(i')] $\Lambda^{-1}h\leq g(t)\leq \Lambda h$;
\item[(ii')] For any $k\in \mathbb{N}$, there is $C(k,n,\Lambda_0)>0$ such that for all $t\in (0,T]$,
\begin{equation*}
\sup_M |\tilde\nabla^k g(t)|\leq C(k,n,\Lambda_0) t^{-k/2};
\end{equation*}
\item[(iii')]$\mathcal{R}(g(t))\geq 0$;
\item[(iv')] $\displaystyle \left(\fint_{B_h(x,1)} |\tilde \nabla g(t)|^n \,d\mathrm{vol}_h \right)^{1/n}\leq C_0\e_0$ for all $(x,t)\in M\times (0,T]$.
\end{enumerate}
The standard rigidity shows that $g(t)$ is Ricci-flat for all $t\in (0,T]$.  By \eqref{eqn: h-flow-Ricciform},
\begin{equation}
\left\{\begin{array}{ll}
\partial_t g_{ij}=\nabla_i W_j +\nabla_jW_i;\\[1mm]
W^k=g^{pq}\left(  \Gamma_{pq}^k-\tilde\Gamma_{pq}^k\right)
\end{array}
\right.
\end{equation}
so that $g(t)=\Psi_t^* g(T)$ where $\Psi_t$ is a diffeomorphism of $M$ satisfying
\begin{equation}
\left\{
\begin{array}{ll}
\partial_t\Psi_t(x)=W(\Psi_t(x),t);\\[1mm]
\Psi_T(x)=x
\end{array}
\right.
\end{equation}
for all $(x,t)\in M\times (0,T]$. By (i'), $\Psi_{t}$ and $\Psi_{t}^{-1}$ are uniformly $C^{1}$ bounded (independent of $t$). Using (iv') and Rellich-Kondrachov theorem, passing to a subsequence, we may assume $g_{i,0}$ converges to some $g_\infty$ in $L^p$ for all $p>0$. By applying Proposition~\ref{timezero-regularity-W1n} on $g_i(t)$ and followed by letting $i\to +\infty$,  a covering argument show that
\begin{equation}
\int_M |\Psi_t^*g(T)-g_{\infty}|^p \, d\mathrm{vol}_h \leq L_{p}t
\end{equation}
for some $L_{p}>0$ independent of $t$. Then
\begin{equation}
\begin{split}
& \int_{M}|(\Psi_{t}^{-1})^{*}g_{i,0}-g(T)|^{p}\,d\vol_{h} \\
\leq {} & C\int_{M}|(\Psi_{t}^{-1})^{*}g_{i,0}-(\Psi_{t}^{-1})^{*}g_{\infty}|^{p}\,d\vol_{h}
+C\int_{M}|(\Psi_{t}^{-1})^{*}g_{\infty}-g(T)|^{p}\,d\vol_{h} \\
\leq {} & C\int_{M}|g_{i,0}-g_{\infty}|^{p}\,d\vol_{h}
+C\int_M |\Psi_t^*g(T)-g_{\infty}|^{p} \, d\mathrm{vol}_h \\
\leq {} & C\int_{M}|g_{i,0}-g_{\infty}|^{p}\,d\vol_{h}+CL_{p}t \\
\end{split}
\end{equation}
for some $C>0$ independent of $t$ and $i$. It then follows that $(\Psi_{i^{-1}}^{-1})^{*}g_{i,0}$ converges to Ricci flat metric $g(T)$ in $L^{p}$ sense.
\end{proof}

\begin{rem}
By working on the diffeomorphic Ricci flow directly with the regularity estimates established in this work,  we believe that one can use the decomposition theorem in \cite{LeeNaberNeumayer2020} to obtain a similar result in term of $d_p$ topology which is more intrinsic in nature.
\end{rem}

\section{Application on singular metric with scalar lower bound}\label{Sec:singular}

In this section, we would like to use the Ricci-Deturck flow to study the following conjecture of Schoen concerning singular metrics with $\mathcal{R}\geq 0$ on manifolds with non-positive Yamabe invariant:
\begin{conj}[Conjecture 1.5 in \cite{LiMantoulidis2019}] \label{conj-1}
Let $M^n$ be a compact manifold with $\sigma(M)\leq 0$. Suppose $g$ is an $L^\infty$ metric on $M$ such that for some  for some smooth metric $h$ and $\Lambda>1$, $\Lambda^{-1}h\leq g\leq \Lambda h$ and is smooth away from a closed, embedded submanifold $\Sigma$ with co-dimension $\geq 3$ and satisfies $\mathcal{R}(g)\geq 0$ outside $\Sigma$, then $\Ric(g)=0$ and $g$ can be extended smoothly on $M$.
\end{conj}

Using the regularization result from Theorem~\ref{thm-existence} together with the maximum principle method in \cite{LeeTam2021},  we have the following partial result.
\begin{thm}\label{rigidity-FollowingLEETAM}
Let $M^n$ be a compact manifold with $\sigma(M)\leq 0$ and $\Sigma$ is a compact set of co-dimension $\geq 3$ on $M$.  Suppose $g_0$ is a $L^\infty\cap W^{1,n}$ Riemannian metric on $M$ so that $g_0\in C^\infty_{\loc}(M\setminus \Sigma)$ and $\mathcal{R}(g_0)\geq 0$ outside $\Sigma$. Then $\Ric(g_0)=0$ outside $\Sigma$. If $\Sigma$ consists of only isolated points, then $g_0$ is a smooth metric with respect to a possibly different smooth structure on $M$.
\end{thm}

\begin{proof}[Sketch of the proof]
It suffices to find an smooth approximation $g_i$ such that $g_i\to g_0$ in $C^\infty_{\loc}(M\setminus \Sigma)$ and $\mathcal{R}(g_i)\geq 0$ on $M$. In \cite{LeeTam2021}, this was carried out in the case when $g_0$ is $C^0(M)$ using the result of \cite{Simon2002}. In the setting of $L^\infty\cap W^{1,n}$, we use the new existence theory instead. Since the argument is almost identical, we only sketch the proof here.

 By Theorem~\ref{shorttime-compact}, we may find a Ricci-Deturck $h$-flow $g(t)$ starting from $g_0$ in $W^{1,n}$ sense such that $g(t)$ stays uniformly equivalent to $h$ and satisfies
 \begin{equation}
 |\tilde \nabla g(t)|^2+ |\tilde \nabla^2 g(t)|+|\Rm(g(t))|\leq \frac{\delta}{t}
 \end{equation}
where $\delta$ can be made as small as we wish by shrinking the time. The curvature estimates follows from the boundedness of $|\Rm(h)|$ and the $C^2$ bound of $g(t)$ with respect to $h$. By the proof of \cite[Theorem 1.1]{LeeTam2021}, we have $\mathcal{R}(g(t))\geq 0$ for $t>0$ where the smallness of $\delta$ and uniform bi-Lipschitz of $g(t)$ are essentially used.  Hence, $\Ric(g(t))\equiv 0$ on $M\times (0,T]$ by the classical rigidity. Since $g(t)$ is smooth up to $t=0$ outside $\Sigma$. We obtain the conclusion by letting $t\to 0$ on $g(t)$ outside $\Sigma$.
\end{proof}

\begin{rem}
We remark here that one can obtain an appropriate extension of $g_0$ across $\Sigma$ using the additional structure $W^{1,n}$ following the method in \cite{LeeTam2021}, see also \cite{LammSimon2021}.  Moreover, the co-dimension of $\Sigma$ can be relaxed further using the concept of upper Minkowski dimension as in \cite{LeeTam2021}. Since we are primarily interested in the conjecture of Schoen on $L^\infty$ metrics, we do not pursue here.
\end{rem}

As one can expect from the works in \cite{LeeTam2021,LiMantoulidis2019,ShiTam2018}, we can use the regularization method in the proof of Theorem~\ref{rigidity-FollowingLEETAM} to prove the following positive mass Theorem. We refer readers to \cite{LeeTam2021,ShiTam2018} for the background setting. For related works, see also \cite{Miao,JiangShengZhang2020,Lee2013,LeeLeFloch2015}.
\begin{thm}\label{rigidity-FollowingLEETAM-PMT}
Let $(M^n, g_0)$ be a Asymptotically flat manifold with $n\geq 3$, $g_0$ is a $L^\infty\cap W^{1,n}_{loc}$ metric on $M$ such that $g_0$ is smooth away from some compact set $\Sigma$ of $M$ of co-dimension at least $3$. Suppose $\mathcal{R}(g_0)\geq 0$ outside $\Sigma$, then the ADM mass of each end is nonnegative. Moreover, if the ADM mass of one of the ends is zero, then $M$ is diffeomorphic to $\mathbb{R}^n$ and $g_0$ is flat outside $\Sigma$.
\end{thm}
\begin{proof}[Sketch or Proof]
The proof follows verbatim from that of \cite[Theorem 1.2]{LeeTam2021}. We only point out the modifications. By Theorem~\ref{shorttime-compact} and the asymptotically flat assumption, there is asymptotically flat Ricci-Deturck $h$ flow on $M$ starting from $g_0$ in $W^{1,n}_{loc}$ sense. We may assume $h$ to be $g_{euc}$ at each ends. Moreover, for each end $E$ and $t>0$,
$$m_{ADM}(E,g(t))\leq m_{ADM}(E,g_0).$$
See the proof of \cite[Proposition 5.1]{LeeTam2021} which is based on \cite{McFeronSze2012}. The proof of \cite[Theorem 1.2]{LeeTam2021} infers that $g(t)$ is of $\mathcal{R}(g(t))\geq 0$ for $t>0$ and hence the positive mass Theorem for smooth metrics \cite{SY,SchoenYau2017} implies that $m_{ADM}(E,g(t))\geq 0$ and hence the $m_{ADM}(E,g_0)\geq 0$. Moreover, if $m_{ADM}(E,g_0)=0$ for some end $E$, then the ADM mass of $g(t)$ along the same end is zero and hence $(M,g(t))$ is isometric to the flat Euclidean space. The flatness of $g_0$ outside $\Sigma$ follows from the smooth convergence outside $\Sigma$.
\end{proof}

\section{Application on scalar curvature persistence}\label{Sec:scalarPres}

In this section, we would like to understand how scalar curvature lower bound behaves when the convergence is weak while the limiting metric is a-priori smooth. Under the uniform bi-Lipschitz assumption, we can show that the scalar curvature lower bound persists under $W^{1,n}_{\loc}$ convergence. We focus on the compact case while the analogous non-compact case can be proved by the similar argument.
\begin{thm}\label{thm:stability-R-lowerbound}
Let $(M,h)$ be a compact manifold. Suppose $g_{i,0}$ is a sequence of smooth Riemannian metrics on $M$ such that
\begin{enumerate}
\item[(i)] $\Lambda_0^{-1}h\leq g_{i,0}\leq \Lambda_0 h$ on $M$ for some $\Lambda_0>1$;
\item[(ii)] $\limsup_{i\to +\infty}\|g_{i,0}-g_0\|_{W^{1,n}(M,h)}=0$ for some smooth metrics $g_0$ and $h$;
\item[(iii)] $\mathcal{R}(g_{i,0})\geq \kappa$ on $M$ for some $\kappa\in \mathbb{R}$.
\end{enumerate}
Then we have $\mathcal{R}(g_0)\geq \kappa$ on $M$.
\end{thm}
\begin{proof}
We follow the strategy in \cite{Bamler2016}. By Theorem~\ref{thm-existence} and Theorem~\ref{shorttime-compact}, for $i$ sufficiently large there exists a solution to the Ricci-Deturck $h$-flow $g_i(t)$ on $M\times [0,T]$ starting from $g_{i,0}$ such that $g_i(t)$ is uniformly smooth away from $t=0$ and stay uniformly equivalent to $h$ for $t\in (0,T]$. As in the proof of Theorem~\ref{thm:stability-sigma}, we obtain a smooth Ricci-Deturck $h$-flow $g(t)$ on $M\times (0,T]$ such that $\mathcal{R}(g(t))\geq \kappa$ for $t\in (0,T]$. Moreover, Proposition~\ref{timezero-regularity-W1n} and our assumptions imply that $g(t)$ attains $g_0$ as the initial data in $L^2$ sense as $t\to 0$. By Theorem~\ref{thm-unique}, $g(t)$ coincides with the classical smooth Ricci-Deturck $h$-flow solution $\hat g(t)$ with initial data $g_0$. This shows that $\mathcal{R}(\hat g(t))\geq \kappa$ for all $t\in (0,T]$. The conclusion follows by letting $t\to 0$ on $\hat g(t)$.
\end{proof}
\begin{rem}
In \cite{Bamler2016}, the persistence of scalar curvature is made local so that $\kappa$ is allowed to be a function, see also \cite{HuangLee2021} for some generalizations along this line. We expect that similar localization under the current setting is also true using the technique in \cite{HuangLee2021}.
\end{rem}


\begin{thebibliography}{10}

\bibitem{Allen2020} Allen, B., {\sl Almost non-negative scalar curvature on Riemannian manifolds conformal to tori}, J Geom Anal (2021). https://doi.org/10.1007/s12220-021-00677-2

\bibitem{AHPPS2019} Allen, B.; Hernandez-Vazquez, L.; Parise, D.; Payne, A.; Wang, S., {\sl Warped tori with almost non-negative scalar curvature}, Geom. Dedicata 200 (2019), 153--171.

\bibitem{Bamler2016} Bamler, R., {\sl A ricci flow proof of a result by gromov on lower bounds for scalar curvature}, Math. Res. Lett. 23 (2016), no. 2, 325--337.

\bibitem{BCRW2019} Bamler, R; Cabezas-Rivas, E;  Wilking, B., {\sl The Ricci flow under almost non-negative curvature conditions}, Invent. Math. 217 (2019), no. 1, 95--126.

\bibitem{Burkhardt2019} Burkhardt-Guim, P., {\sl Pointwise lower scalar curvature bounds for $C^0$ metrics via regularizing Ricci flow}, Geom. Funct. Anal. 29 (2019), no. 6, 1703--1772.

\bibitem{CAP2020} Cabrera Pacheco, A.-J.; Ketterer, C.; Perales, R., {\sl Stability of graphical tori with almost nonnegative scalar curvature}, Calc. Var. Partial Differential Equations 59 (2020), no. 4, Paper No. 134, 27 pp.

\bibitem{CalabiHartman1970} Calabi, E.; Hartman, P., {On the smoothness of isometries}, Duke Math. J. 37 (1970), 741--750.

\bibitem{ChanChenLee2022} Chan, P.-Y.; Chen, E.; Lee, M.-C.,{\sl Small curvature concentration and Ricci flow smoothing}, J. Funct. Anal. 282 (2022), no. 10, Paper No. 109420, 29 pp.

\bibitem{ChauLee2020} Chau, A.; Lee, M.-C., {The \K Ricci flow around complete bounded curvature \K metrics},  Trans. Amer. Math. Soc. 373 (2020), no. 5, 3627--3647.

\bibitem{ChuLee2021} Chu, J.; Lee, M.-C., {\sl Conformal tori with almost non-negative scalar curvature}, preprint, arXiv:2103.07003, to appear in Calc. Var. Partial Differential Equations.

\bibitem{ChuLee2021-2} Chu, J.; Lee, M.-C., {\sl K\"ahler tori with almost non-negative scalar curvature}, preprint, arXiv:2112.04834, to appear in Commun. Contemp. Math.

\bibitem{Deturck1983} DeTurck, D. M., {\sl Deforming metrics in the direction of their Ricci tensors}, J. Differential Geom. 18 (1983), no. 1, 157--162.

\bibitem{GreenWu1979} Greene, R. E.; Wu, H., {\sl $C^\infty$ approximations of convex, subharmonic, and plurisubharmonic functions}, Ann. Sci. \'Ecole Norm. Sup. 12 (1979), 47--84.

\bibitem{Gromov2014} Gromov, M., {\sl Dirac and Plateau billiards in domains with corners}, Cent. Eur. J. Math. 12 (2014), no. 8, 1109--1156.

\bibitem{GromovLawson1980} Gromov, M.; Lawson, H. B., Jr., {\sl Spin and scalar curvature in the presence of a fundamental group. I}, Ann. of Math. (2) 111 (1980), no. 2, 209--230.

\bibitem{Hamilton1982} Hamilton, R. S., {\sl Three-manifolds with positive Ricci curvature}, J. Differential Geometry 17 (1982), no. 2, 255--306.

\bibitem{Hochard2016} Hochard, R., {\sl Short-time existence of the Ricci flow on complete, non-collapsed 3-manifolds with Ricci curvature bounded from below}, preprint, arXiv:1603.08726.

\bibitem{HuangLee2021} Huang, Y.; Lee, M.-C., {\sl Scalar curvature lower bound under integral convergence}, preprint, arXiv:2111.05079.

\bibitem{HuangTam2018} Huang, S.-C.; Tam, L.-F., {\sl K\"ahler-Ricci flow with unbounded curvature}, Amer. J. Math. 140 (2018), no. 1, 189--220.



\bibitem{JiangShengZhang2020} Jiang, W.; Sheng, W.;  Zhang, H., {\sl Removable singularity of positive mass theorem with continuous metrics}, arXiv:2012.14041


\bibitem{JiangShengZhang2021} Jiang, W.; Sheng, W.; Zhang, H., {\sl Weak scalar curvature lower bounds along Ricci flow}, preprint, arXiv:2110.12157.

\bibitem{KochLamm2012} Koch, H.; Lamm, T, {\sl Geometric flows with rough initial data}, Asian J. Math. 16 (2012), no. 2, 209--235.

\bibitem{Lai2019} Lai, Y., {\sl Ricci flow under local almost non-negative curvature conditions}, Adv. Math. 343 (2019), 353--392.

\bibitem{Lai2021} Lai, Y., {\sl Producing 3D Ricci flows with nonnegative Ricci curvature via singular Ricci flows}, Geom. Topol. 25 (2021), no. 7, 3629--3690.

\bibitem{LammSimon2021} Lamm, T.; Simon, M., {\sl Ricci flow of $W^{2,2}$-metrics in four dimensions}, preprint, arXiv:2109.08541.


\bibitem{Lee2013}Lee, D. A., {\sl A positive mass theorem for Lipschitz metrics with small singular sets}, Proc. Amer. Math. Soc. 141 (2013), no. 11, 3997--4004.
\bibitem{LeeLeFloch2015}Lee, D. A.; LeFloch, P. G., {\sl The positive mass theorem for manifolds with distributional curvature}, Comm. Math. Phys. 339 (2015), no. 1, 99--120.


\bibitem{LeeTam2020} Lee, M.-C.; Tam, L.-F., {\sl Chern-Ricci flows on noncompact manifolds}, J. Differential Geom. 115 (2020), no. 3, 529--564.

\bibitem{LeeTam2021-GT} Lee, M.-C.; Tam, L.-F., {\sl K\"ahler manifolds with almost non-negative curvature}, Geom. Topol. 25 (2021), no. 4, 1979--2015.

\bibitem{LeeNaberNeumayer2020} Lee, M.-C.; Naber, A.; Neumayer, R., {\sl $d_p$ convergence and $\e$-regularity theorems for entropy and scalar curvature lower bounds}, preprint, arXiv:2010.15663, to appear in Geom. Topol.

\bibitem{LeeTam2021} Lee, M.-C.; Tam, L.-F., {\sl Continuous metrics and a conjecture of Schoen}, preprint, arXiv:2111.05582.

\bibitem{LiBook} Li, P., {\sl Geometric analysis}, Cambridge Studies in Advanced Mathematics, 134. Cambridge University Press, Cambridge, 2012. x+406 pp. ISBN: 978-1-107-02064-1

\bibitem{LiMantoulidis2019} Li, C.; Mantoulidis, C., {\sl Positive scalar curvature with skeleton singularities}, Math. Ann. 374 (2019), no. 1-2, 99--131.



\bibitem{McFeronSze2012} McFeron, D.; Sz\'ekelyhidi, G., {\sl On the positive mass theorem for manifolds with corners}, Comm. Math. Phys. 313 (2012), no. 2, 425--443.


\bibitem{Miao}Miao, P., {\sl Positive mass theorem on manifolds admitting corners along a hypersurface}, Adv. Theor. Math. Phys. 6 (2002), no. 6, 1163--1182 (2003).

\bibitem{SSS2008} Schn\"urer, O. C.; Schulze, F.; Simon, M., {\sl Stability of Euclidean space under Ricci flow}, Comm. Anal. Geom. 16 (2008), no. 1, 127--158.

\bibitem{Schoen1989} Schoen, R., {\sl Variational theory for the total scalar curvature functional for Riemannian
metrics and related topics}, Topics in calculus of variations (Montecatini Terme, 1987), 120--154, Lecture Notes in Math., 1365, Springer, Berlin, 1989.

\bibitem{SY}Schoen, R. M.; Yau, S.-T., {\sl On the proof of the positive mass conjecture in general relativity}, Comm. Math. Phys. 65 (1979), no. 1, 45–76.

\bibitem{SchoenYau1979} Schoen, R.; Yau, S.-T., {\sl Existence of incompressible minimal surfaces and the topology of three-dimensional manifolds with nonnegative scalar curvature}, Ann. of Math. (2) 110 (1979), no. 1, 127--142.

\bibitem{SchoenYau1979-2} Schoen, R.; Yau, S.-T., {\sl On the structure of manifolds with positive scalar curvature}, Manuscripta Math. 28 (1979), no. 1-3, 159--183.

\bibitem{SchoenYau2017}Schoen, R. M.; Yau, S.-T., Positive Scalar Curvature and Minimal Hypersurface Singularities, arXiv:1704.05490.


\bibitem{Shi1989} Shi, W.-X., {\sl Deforming the metric on complete Riemannian manifold}, J. Differential Geom. 30 (1989), no. 1, 223--301.

\bibitem{ShiTam2018} Shi, Y.; Tam, L.-F., {\sl Scalar curvature and singular metrics}, Pacific J. Math. 293 (2018), no. 2, 427--470.

\bibitem{Simon2002} Simon, M., {\sl Deformation of $C^0$ Riemannian metrics in the direction of their Ricci curvature}, Comm. Anal. Geom. 10 (2002), no. 5, 1033--1074.

\bibitem{SimonTopping2017} Simon, M.; P.-M. Topping., {\sl Local mollification of Riemannian metrics using Ricci flow, and Ricci limit spaces}, Geom. Topol. 25 (2021), no. 2, 913--948.

\bibitem{Sormani2021} Sormani, C., {\sl Conjectures on Convergence and Scalar Curvature}, preprint, arXiv:2103.10093.

\bibitem{Tam2010} Tam, L.-F., {\sl Exhaustion functions on complete manifolds}, Recent advances in geometric analysis, 211--215, Adv. Lect. Math. (ALM), 11, Int. Press, Somerville, MA, 2010.

\bibitem{ToppingHao2021} Topping P. M.; Hao Y., {\sl Smoothing a measure on a Riemann surface using Ricci flow}, preprint, arXiv: 2107.14686.

\end{thebibliography}
\end{document}